\documentclass[11 pt]{article}
\usepackage{amsmath,amssymb}
\usepackage{amsfonts}
\usepackage{amsthm}
\usepackage{mathrsfs}
\usepackage[shortlabels]{enumitem}

\swapnumbers
\newtheorem{thm}{Theorem}[section]
\newtheorem{lem}[thm]{Lemma}
\newtheorem{prop}[thm]{Proposition}
\newtheorem{corr}[thm]{Corollary}

\newtheorem{fact}[thm]{Fact}
\theoremstyle{definition}
\newtheorem{remark}[thm]{Remark}
\theoremstyle{definition}
\newtheorem{defin}[thm]{Definition}
\theoremstyle{definition}
\newtheorem{secc}[thm]{}
\theoremstyle{definition}
\newtheorem{comm}[thm]{Comment}
\newtheorem{example}[thm]{Example}
\def\ZZ{\mathbb Z}
\begin{document}

\title { {\hspace{12mm}} F-Noetherian Rings and Skew $\newline$  Quantum Ring Extensions}

\vspace{15mm}
\vspace{15mm}
\author{Nazih Nahlus}
\date{}
\maketitle
\vspace{6mm}
\begin{center}\emph{Dedicated to George Bergman and Donald Passman}\end{center}
\vspace{3em}

\begin{abstract}
A ring $R$ shall be called F-noetherian if every finite subset of $R$ is contained in a (left and right) noetherian subring of $R$. 
F-noetherian rings have many interesting linear algebra properties which we refer to as the full strong rank condition, fully stably finite, and more generally the basic condition.  We also study some basic ring-theoretic properties of F-noetherian rings such as localizations of F-noetherian rings. The F-noetherian property is preserved under some  \emph{skew} quantum ring extensions
 including some iterated Ore extensions and some quantum almost-normalizing extensions. For example, let $R= S[ x_1, \dots, x_n ]$ be a finitely generated ring \textit {over a subring $S$} such that  (1) for 
\textbf{$i < j$} ,  
 \[ x_j x_i -q_{ji} x_i x_j  \in S [ x_1, ..., x_{j-1}] + Sx_j , \textit {or} \]
\[ x_j x_i -q_{ji} x_i x_j  \in S x_1+  ... + Sx_n + P_2 \] 
for some units $q_{ji} \in S$ and $P_2$  is a certain set of  quadratic polynomials (related to the quantum group $\mathcal{O}_q(G)$ where $G$ is a connected complex semisimple algebraic group), (2) for all $i$,
$ Sx_i +S= S+ x_i S, \text {and} $
(3) each $x_i$ commutes with a subring $A$ of $S$ such that $S$ is finitely generated 
as a ring over $A$. Then if $S$ is noetherian, or F-noetherian, or a direct limit of noetherian rings, so is $R$.
\end{abstract}
\section{Introduction} \label{intro}

In this paper we introduce the new notion of an F-noetherian ring and related notions.  A ring $R$ shall be called F-noetherian if every finite subset of $R$ is contained in a (left and right) noetherian subring of $R$.  For example, by the Hilbert Basis Theorem, every commutative ring is tightly F-noetherian  in the sense that every finite subset of $R$ generates a noetherian subring of $R$.  Tightly F-noetherian rings are directed F-noetherian in the sense that for every finite subset $A$  of $R$, there exists a noetherian subring $R(A)$ of $R$ containing $A$ such that if  $X \subset Y$ are finite subsets of $R$, then $R(X) \subset R(Y)$. It turns  out that  the directed F-noetherian rings 
are precisely  the direct limits of noetherian rings  by Proposition 8.3.

An F-noetherian ring $R$  has many interesting linear algebra properties.  For example,  
\begin{enumerate}[(i)]
\item if $M^m$ is embedded in $M^n$ and $M$ is a finitely generated non-zero $R$-module, then $m\le n$.   More generally, for (right or left) $R$-modules $M$ and $N$ where $M$ is finitely generated, $M \oplus N$ can not be imbedded in $M$ unless $N=0$, or equivalently, every monomorphism of a finitely generated $R$-module $M$ has an essential image. For convenience, this condition will be referred to as the full SRC (strong rank condition) 
\item every epimorphism $ M \rightarrow M$ of a finitely generated $R$-module $M$ is an isomorphism. That is, every finitely generated (right or left) $R$-module $M$ is Hopfian. For convenience, this condition will be referred to as the fully stably finite condition.
\end{enumerate}

The above two linear algebra properties are special cases of the following more general property.
Every F-noetherian ring $R$ is \textbf{\emph{basic}} in the sense that if we have an epimorphism
$ f : A \rightarrow M$ and a monomorphism $i: A \rightarrow M$ �of (right or left) $R$-modules and $M$ is a finitely generated $R$-module, then $f$ is an isomorphism (and $i(A)$ is essential in $M$). This property was suggested to us by an exercise on commutative and noetherian rings, in [L 3, Exercise 1.10] which is taken from p. 61 of a book entitled "Commutative Noetherian rings and Krull Rings" by Balcerzyk and J\'{o}sefiak. (See section 2).

\vspace{3mm}
In section 3, we study some ring-theoretic properties of F-noetherian rings. For example, we have the following.

\begin{remark}.
\begin {enumerate}[(i)]
\item F-noetherian rings and tightly F-noetherian rings are preserved under direct limits and homomorphic images. 
\item If $R$ and $S$ are tightly F-noetherian (resp. F-noetherian) rings, then so is $R\times S$.  
\item If $R$ is an F-noetherian domain, then $R$ is an Ore domain. 
\item Let $R$ be an F-noetherian ring and let $S$ be a multiplicative subset of $R$.  Suppose there exists a finite set $Q$ in $R$ such that  $sR'=R's$  for every element $s \in S$ and  every subring $R'$ of $R$ containing both $S$ and $Q$.  Then $S$ is a denominator set of $R$ and the localization ring $RS^{-1}$ is F-noetherian.    

\item If $R$ has a non-negative filtration whose associated graded ring gr$(R)$ is F-noetherian,  then the ring $R$ may fail to be F-noetherian. This failure may even occur if gr$(R)$ is commutative. 
\end{enumerate}
For more details about the above facts, see (3.6)--(3.9) and (5.10).
\end{remark}

\vspace{1mm}
In section 4, we study the preservation of the F-noetherian property under quantum iterated Ore extensions and Laurent ring extensions. We shall see that if $S$ is an F-noetherian ring, 
then its differential skew polynomial rings may \emph{fail}  to be F-noetherian. However, we have the following.
\begin{thm} Let $R = S[x_1; f_1, d_1][x_2; f_2, d_2] \dots [x_n; f_n, d_n]$  be an iterated Ore extension over a ring $S$  such that
\begin{enumerate}[(1)]
\item for each $j>i$,  $f_j(x_i) =� q_{ji}(x_i)$ for some unit $q_{ji}�\in S$,
\item each $f_i$ is the identity on $S$, 
\item we have one of the following cases.
\begin{description}
\item[Case 1:] For all i,  $d_i(S) \subset S$, and  \\
all $d_i$ are commuting locally nilpotent derivations on $S$.
\item[Case 2:] $S$ has an ascending chain of subsets $A_m$ with $m \in \mathbb{N}$, such that
$A_0=0, S$ is the union of all $A_m$, and
             $$d_i(A_m) \subset A_{m-1}$$  for all  $ m \in \mathbb{N}$ , and  $i=1,2, \dots n$.
\end{description}
\end{enumerate}
Then,
\begin{enumerate}[(i)]
\item if $S$ is F-noetherian, so is $R$, and
\item if $S$ is tightly F-noetherian, then $R$ is directed F-noetherian.
\end{enumerate}
Note that we have no restrictions on each $d_j (x_i)$  for $j>i$
\end{thm}

Similarly, in the case of skew-Laurent rings, if $S$ is F-noetherian, 
then $R=S[x, x^{-1}; f]$  may fail to be F-noetherian. However, we have

\begin{prop} Let $R=S[x_1, x_1^{-1}; f_1]\dots[x_n, x_n^{-1}; f_n]$ 
be an iterated skew-Laurent ring extension such that 
\begin{enumerate}[(1)]
\item for each $j>i$,��$f_j(x_i) =� q_{ji}(x_i)$ for some unit $q_{ji}�\in S$,
\item each $f_i$  is the identity on $S$.
\end{enumerate}
 If $S$ is F-noetherian, then so is $R$.
\end{prop}

In section 5, we study the preservation of the F-noetherian property under many types of skew quantum ring extensions starting with almost centralizing extensions in the sense of Passman in [P]. The first general result in this section is Theorem 1.4 below (stated in  Theorem 5.4) in which the noetherian part (i) in case 3  is a simple generalization of  Proposition I.8.17 in  [B-G, p. 77] which is one of the key steps in proving that the quantum group  $\mathcal{O}_q(G)$  is noetherian (where $G$ is a connected complex semisimple algebraic group and $q$ is a fixed non-root of unity). Specifically, Proposition I.8.17 in [B-G] assumes (versus our generalization below) that each $x_i$ commutes with the elements of $S$, it assumes our $f=0$ and each $(x_j x_i� - q_{ji}� x_i  x_j)�=0$ if $i=1$. Finally, the noetherian part(i) in case 2, is a generalization of [L-R, Cor. 2.4] since we do not assume any PBW $S$-basis of $R$.  

\begin{thm}\label{5.4}
Let $R = S [ x_1, \dots , x_n ] $ be a finitely generated ring over a subring $S$ satisfying the following two  conditions.
\begin{enumerate}[(1)]
\item for all $i$, �$x_i S  + S = S + S x_i$.
\item for all \emph { $ j > i$},� there exist units $q_{ji}� \in S$ such that we have one of the following cases.
\begin{description}
\item[Case 1:] $(x_j x_i� - q_{ji}� x_i  x_j)� \in   S [ x_1, \dots, x_{j-1} ] + Sx_j $
\item[Case 1':] $ (x_j x_i� - q_{ji}� x_i  x_j)� \in   Sx_i + S [ x_{i+1}, \dots , x_n ] $

\item[Case 2:] $(x_jx_i� - q_{ji}� x_i x_j)�  \in  S + Sx_1+ \cdots + Sx_n$

\item[Case 3:]  $x_j x_i� - q_{ji} x_i x_j) = f + g$  where $f  \in S + Sx_1+ \cdots + Sx_n$ and if $i >1 $, $g$ is a finite sum of quadratic monomials  $s x_a x_b $ where $s \in S$ and either  $a$ or $b$ is at most $i-1$; however if $i=1$,  $g$ is a finite sum of quadratic monomials $s x_a x_b$ where $s \in S$ and $a=1$ and $b<j$ or vice versa ($b=1$ and $a<j$).

\end{description}
\end{enumerate}
Then in all cases, we have the following.
\begin{enumerate}[(i)]
\item if $S$ is noetherian, then so is $R$.
\item Now we suppose that each $x_i$ commutes with the elements  of a subring $A$ of $S$ such that $S$ is finitely generated as a ring over $A$.
\\ Then, if $S$ is F-noetherian, so is $R$.
\item Under the additional assumption in (ii),  if $S$ is directed  F-noetherian, then so is $R$.
\end{enumerate} 
\end{thm}

\vspace {1mm}
Then we mix the cases of Theorem 1.4  in two different ways that are stated in Theorems 5.8 and 5.9. 
For example, Theorem 5.8 says the following.

\begin{thm} Let  $G= k[t_1,  \dots , t_m,  x_1,  \dots , x_n]$  be a finitely generated ring over its subring k  such that
\begin{enumerate}[(i)]
\item For all $j > i$, there exist units $p_{ji} \in k$  such that   
\[t_j t_i  - p_{ji}  t_i t_j  \in   k[t_1, \dots, t_{j-1}] + kt_j\]
\item For all $j > i$, there exist units $q_{ji} \in k$  such that 
\begin{description}
\item[Case 1: ]   $(x_jx_i  - q_{ji}  x_i x_j)  \in   k[x_1, \dots, x_{j-1}] + kx_j $
\item[Case 2: ] $(x_j x_i  - q_{ji}  x_i  x_j)   \in  k + kx_1+ \cdots + kx_n$
\item[Case 3: ]  $(x_j x_i� - q_{ji} x_i x_j) = f + g$  where $f  \in S + Sx_1+ \cdots + Sx_n$ and if $i >1 $, $g$ is a finite sum of quadratic monomials  $s x_a x_b $ where $s \in S$ and either  $a$ or $b$ is at most $i-1$; however if $i=1$,  $g$ is a finite sum of quadratic monomials $s x_a x_b$ where $s \in S$ and $a=1$ and $b<j$ or vice versa ($b=1$ and $a<j$). 

\end{description}
\item For all $j$ and $i$, there exist units $c_{ji} \in k$ such that 
\[t_ix_j  - c_{ji} x_j  t_i  \in  k + kt_1+ \cdots + k t_u+ kx_1+ \cdots + kx_v\]
where in case 1, $(u, v)=(i, j-1)$ or $(i-1, j)$; while  in case 2, $(u, v)=(i, n)$; and in case 3, $(u, v)=(i-1, j)$.
\end{enumerate}  
 
Moreover, 

\begin{enumerate}[(1)]
\item for all $i,j$, $ k t_i +k = k + t_i k$ and $k x_j +k = k + x_j k$, and 
\item each $t_i$ and each $x_i$ commute with the elements of a subring $A$ of $S$
such that $S$ is finitely generated as a ring over $A$.

\end{enumerate}
Then in all cases, if $k$ is noetherian or F-noetherian or directed F-noetherian,  so is $G$.
\end{thm}

\vspace{2mm}
In section 6, we shall give many examples. For example, we generalize the Hayashi example in (6.5) as follows.
\example  \textit {A generalization of the Hayashi Example}.  Let $R$ be the $k$-algebra generated by the variables $x_i, y_i, z_i$,    $1 \leq i \leq n$ such that any pair of variables almost-commute (in the sense that $ab=q_{ab} ba$ where $q_{ab}$ is a  unit in $k$) except for the pairs $(x_i, z_i)$ where we have the relations  
                       
                                        \[ (z_i x_i + qx_i z_i) y_i  = 1 =   y_i (z_i x_i + qx_i z_i) \]

for some unit $q \in S$.  Then if $k$ is noetherian or F-noetherian, so is $R$.

\vspace {3mm}
In section 7, we generalize many quantum groups as in the following example stated in (7.2).
\begin{example} \it {A generalization of the quantum group $\mathcal{O}_q(M_n(k))$}

\vspace{7mm}
\indent  Let   $ G(\mathcal{O}_q (M_n(k))) =  k[t_1,  \dots , t_m,  x_{11},  \dots , x_{nn}]$ be the k-algebra generated by the variables $\{t_1,  \dots , t_m,  x_{11},  \dots , x_{nn} \}$ such that 
\begin{enumerate}[(i)]
\item $k[x_{11},  \dots , x_{nn}] = \mathcal{O}_q(M_n(k)) $ (with $q$ being a central unit of $k$) is the $k$-algebra with the standard relations of the quantum group of $\mathcal{O}_q(M_n(k)) $ of  $n \times n$ matrices.  \textit{See [B-K, p. 16]}. 

\item For all $j > i$,  there exist units $p_{ji} \in k$  such that   
\begin{description}
\item[Case 1: ]  $ t_j t_i� - p_{ji}� t_i t_j� \in   k[t_1, \dots, t_{j-1}] + kt_j $
\item [Case 2: ]  $ t_j t_i� - p_{ji}� t_i t_j� \in   k +k t_1+  \cdots + k t_m $
\end{description}

\item Let $\{x_1,  \dots , x_{n^{2}}\}$ be the lexicographic ordering of $ \{x_{11},  \dots , x_{nn}\}$. 
For all $j$ and $i$, we also assume that there exist units $c_{ji}$ in $k$ such that 
\[t_ix_j� - c_{ji}�x_j  t_i  \in  k + kt_1+ \cdots + k t_u+ kx_1+ \cdots + kx_v\]
\end{enumerate}
where in case 1,  $(u, v)=(i, j-1)$ or $(i-1, j)$, while in case 2, $(u,v)= (m,n^{2})$  

\vspace {3mm}
If $k$ is noetherian, or F-noetherian, or directed F-noetherian, 
then so is $ G(\mathcal{O}_q(M_n(k))) $. 
\end{example}

\vspace{3mm}
In section 8, we shall give many examples of F-noetherian matricial rings. For example, we show that $M_2(Z[x])$ is noetherian but not tightly F-noetherian. \\
We also give many examples of F-noetherian group rings. For example, 
let $K[G]$  be the group algebra of a locally finite group $G$ over an F-noetherian ring $K$. (Recall that a group $G$ is locally finite group if every finitely generated subgroup is finite).   (Interesting examples of $G$ are the finitary symmetric/alternating groups on an infinite set). Or let $K[G]$  be the group algebra of a polycyclic-by-finite group $G$ (for example $G$ is a finitely generated nilpotent group)  over an F-noetherian ring $K$. Then in both cases  $K[G]$ is F-noetherian. (See 8.2.4 and 8.2.5)
   
\vspace{4mm}
Finally, in section 9, we pose few open problems. For example,
\vspace {2mm}
\\ {\textbf  {Problem} 1.}   Find an example of an F-noetherian ring which is not a 
\\ direct limit of noetherian rings.
Or equivalently, find an example of an F-noetherian ring which is not directed F-noetherian. 

\vspace{3mm}  
 \textit  {Throughout the paper, we shall use the notation 
$$R=S[x_1,x_2, \dots, x_n]$$ for the ring generated by a subring $S$ and the elements $x_1,  \dots , x_n$ of $R$}. \\

\paragraph{Acknowledgements}

It is my great pleasure to acknowledge my deep indebtedness to George Bergman and Donald Passman for their constant encouragement and invaluable help in preparing this paper. I also wish to thank Kenneth Goodearl and Kenneth A. Brown for some interesting comments on the paper, and to thank Farid Kourki for providing me with some references on section 2.

\vspace{3em}
\section{Linear Algebra over Commutative and Noetherian rings} \label{linAlg}

\vspace{2mm}
\textit {Every ring $R$  has an identity element $1$, all subrings of $R$ have $1$,  all R-modules are unital, and all ring morphisms from $R$ to $S$ take $1_R$ into $1_S$. Moreover, recall that a ring is noetherian if it is left and right noetherian .}  \\
  \\
\par First we recall the following definitions in [L 2, Chapter 1].
\begin{defin}\label{src} A ring $R$ satisfies the (right) \textbf{strong rank condition} (SRC) if, whenever we have a monomorphism of right $R$-modules  $f: R^m \rightarrow R^n$, then $m \le n$.  Equivalently, any set of linearly independent elements in $(R^n)_R$ has cardinality  $\leq n$. We shall reserve the term strong rank condition (SRC) for both right SRC and left SRC.
\end{defin} 
\begin{defin} A ring $R$ is \textbf{stably finite} if every epimorphism of (right or left) $R$-modules  
$f: R^n \rightarrow R^n$  is an isomorphism.  Equivalently, any generating set of $n$ elements of $R^n$ is a basis of $R^n$. Or equivalently, the matrix rings $M_n(R)$ are Dedekind-finite, i.e, they satisfy the property:  $ab=1$ implies $ba=1$.
\end{defin}
For the basic properties of such rings, see [L 2, Chapter 1].
\begin{fact}
Commutative rings and noetherian rings have the SRC
\end{fact}
The Strong Rank Condition (SRC) for commutative rings is a fact that does not seem to be well known as it should be, as remarked in  [L 2, p. 15]). Moreover, every (left and right) noetherian ring has the strong rank condition (SRC). This is proved in a simple way in [L 2, p. 14]. For another proof, one can take the (Goldie) uniform dimension in the embedding $R^m \rightarrow  R^n$, to get  $m \dim(R) \leq n \dim(R)$.
Thus $m \leq n$ (since $R$ is a non-zero noetherian ring). The proof in the commutative case can be easily reduced to the noetherian case by the Hilbert Basis Theorem. 

\begin{secc} We will see that commutative rings and noetherian rings satisfy the following three further properties, for which we shall give specific names for easy reference.
\begin{enumerate}[(i)]
\item \textbf{Full SRC}.  A ring $R$ has the \emph{full SRC condition} if, for (right or left) $R$-modules $M$ and $N$ where $M$ is finitely generated, $M \oplus N$ can not be imbedded in $M$ unless $N=0$. In particular, if $M^m$ is embedded in $M^n$ and $M$ is a finitely generated non-zero $R$-module, then $m \leq n$. Equivalently, a ring $R$ has the full SRC if every monomorphism of a finitely generated $R$-module $M$ has an essential image.
\item \textbf{Fully stably finite}.  A ring $R$ is called \emph{fully stably finite}  if every epimorphism 
$M \rightarrow M$ of a finitely generated $R$-module $M$ is an isomorphism. That is, every finitely generated (right or left) $R$-module $M$ is Hopfian.
\item \textbf{Basic}.  A ring $R$ is called  \emph{basic} if it has the following very interesting property. Suppose we have an epimorphism $f: A \rightarrow M$ and a monomorphism $i: A \rightarrow M$ �of (right or left) $R$-modules.�If $M$ is a finitely generated $R$-module, then $f$ is an isomorphism.  
Or equivalently, by (2.6) below,  $R$ is a basic ring if it has the following (seemingly a bit stronger) property.  Suppose we have an epimorphism $f: A \rightarrow M$ and a monomorphism $i: A \rightarrow M$ �of (right or left) $R$-modules.�If $M$ is a finitely generated $R$-module, then $f$ is an isomorphism and $i(A)$ is \emph{essential} in $M$. 
\end{enumerate}

Note that the basic  property above resembles ``the Schr\"{o}der Bernestein Theorem" for injective $R$-modules [L 2 or 3, Ex. 3.31]. Observe, in the above setting, that if $A$ is finitely generated (instead), then $M$ is finitely generated via $f$. 
\end{secc}
\begin{comm} The above notion of basic rings was motivated to us by a very interesting exercise (stated below) which has many noteworthy special cases as described in [L 3, Ex. 1.10, p. 7] which is taken from p. 61 of the book of Balcerzyk and Josefiak, ``Commutative Noetherian  Rings and Krull Rings," Halsted Press/Polish Sci. Publishers, 1989.

\paragraph{Exercise:}  Suppose we have an epimorphism $f: A \rightarrow M$ and a monomorphism 
$i: A \rightarrow M$  of $R$-modules.  Then $f$ is an isomorphism in the following cases.
\begin{enumerate}
\item $A$ is a noetherian module 
\item $R$ is commutative and $M$ is finitely generated
\end{enumerate}
\end{comm}

\begin{thm}\leavevmode
\begin{enumerate}[(i)]
\item Basic $\Longrightarrow$ full SRC and fully stably finite.
\item A ring $R$ is basic if it has the following (seemingly stronger) property:
Suppose we have an epimorphism $f : A \rightarrow M$ and a �monomorphism 
$i: A \rightarrow M$  of $R$-modules.  If $M$ is a finitely generated $R$-module, then $f$ is an isomorphism and $i(A)$ is essential in $M$.
\item Commutative rings and noetherian rings are basic.
\end{enumerate}
\end{thm}
\begin{proof}\leavevmode
\begin{enumerate}[(i)]
\item Let $R$ be a basic ring. To prove $R$ has the full SRC, let $M$ be a finitely generated $R$-module suppose $M \oplus A$ is embedded in $M$. Then consider the projection 
$f: M \oplus A$ onto $M$. Since $R$ is basic, the projection map $f$ must be an isomorphism. Hence $A=0$, so $R$ has the full SRC.

Now let $f: M \rightarrow M$ be an epimorphism of a finitely generated $R$-modules and consider the identity morphism $M \rightarrow M$. Since $R$ is basic, $f$ must be an isomorphism, and $R$ is fully stably finite. 
\item Assume $R$ to be basic and suppose we have a monomorphism $i: A \rightarrow M$ and an epimorphism $f : A \rightarrow M$ of (right or left) $R$-modules where $M$ is a finitely generated $R$-module. Since $R$ is basic, $f$ is an isomorphism, so $A$ and $M$ are isomorphic. To prove $i(A)$ is essential in $M$, suppose $i(A) \oplus B$ is embedded in $M$. But $i(A), A$, and $M$ are isomorphic. Hence $M \oplus B$ is embedded in $M$, so $B=0$ by (i).  Hence $i(A)$ is essential in $M$.
\item Now suppose we have an epimorphism $f : A \rightarrow M$ and a �monomorphism 
$i: A \rightarrow M$  of $R$-modules where $M$ a finitely generated $R$-module.  If $R$ is commutative, then $f$ is an isomorphism by Exercise (2.5), so $R$ is basic.  If $R$ is noetherian, then the fintely generated $R$-module $M$ is a noetherian, so $i(A)$ and A are noetherian $R$-modules. Hence Exercise (2.5) applies to obtain that $f$ is an isomorphism. 
Hence $R$ is basic. \end{enumerate} \end{proof} 
\par   
\textbf {Note}. After writing the paper, F. Kourki pointed out that rings with the  \textit {full SRC} are characterized (without names)  in [Hag-V, Thm. 3.3]. Moreover, our \textit {fully stably finite} rings are characterized (without names) via their matrix rings in [G, Thm. 7]. Furthermore, our \textit {basic} rings are precisely the left and right II$_1$ rings in [D] where such rings are defined as follows: every epimorphism from a submodule of a finitely generated $R$-module $M$ onto $M$ is an isomorphism. The main result in [D] is that Left II$_1$ rings are closed under direct limits [D, Thm. 2]. Finally, certain P.I. rings are left II$_1$ (see [A-F-S, Thm. 2.2]).
 
\vspace{2em}
\section{F-Noetherian Rings} \label{noeth}

\textit  {Our modules will be viewed as right modules and the arguments follow similarly for left modules}.
\vspace{5mm}

First we introduce the following new definition.
\begin{defin} A ring $R$ is called F-noetherian (for finitely noetherian) if every finite subset of $R$ is contained in a (left and right) noetherian subring of $R$. In other words, a ring $R$ is an F-noetherian ring if $R$ is a union of (not necessarily directed) noetherian subrings $R_i$ with the additional property that every finite subset of $R$ is contained in some $R_i$.
\end{defin}
\begin{defin} A ring $R$ is called \textit {directed} F-noetherian if for every finite subset $A$ of $R$, there exists a noetherian subring $R(A)$ of $R$ containing $A$ such that
if  $X \subset Y$ are finite subsets of $R$, then $R(X) \subset R(Y)$.
\end{defin}

\begin{defin} A ring $R$ is called \textit {tightly} F-noetherian if every finite subset of $R$ generates (as a ring) a noetherian subring of $R$. Note that in this case, $R$ is directed F-noetherian since $R$ is the union of the subrings $R_i$ which are generated by the finite subsets of $R$. 
\end{defin} Hence we have the following trivial implications.
 \[\text{Tightly }F\text{-noetherian} \Longrightarrow \text{ directed }F\text{-noetherian�}\Longrightarrow F\text{-noetherian} \]

\vspace {2mm} 
Moreover, in Proposition (8.3) we will see that 
 \[ \text{Directed F-noetherian rings}  \equiv  \text{Direct limits of noetherian rings}  \]

\begin{thm} F-noetherian rings are basic rings. In particular, F-noetherian rings have the full SRC and the fully stably finite condition.
\end{thm}
\begin{proof}  By (2.7) it suffices to prove that an F-noetherian ring $R$ is basic 
where $M$ is a finitely generated $R$-module. We must prove that $f$ is an isomorphism.

In case $R$ is commutative, one can easily reduce the proof about $f$ being an isomorphism to the noetherian case verbatim as done in the solution to Exercise (1.4)(ii) in [L 3, Ex. 1.10, p. 7]. We comment that this reduction resembles the reduction to the noetherian case done by Strooker in the mid 60's in proving that, if $R$ is a commutative ring, every epimorphism of $R^n$ is an isomorphism.
In fact, this last property was also proved later by Vasconcelos.  However the solution to Exercise (1.4)(ii) in [L 3, p. 7] only requires that every finite subset of $R$ is contained in a noetherian subring of $R$, which is exactly what we have since $R$ is F-noetherian. Hence $f$ is an isomorphism. 
\end{proof}
\begin{example}\label{3.4} The first examples of F-noetherian rings are commutative rings which are also tightly F-noetherian via the Hilbert Basis Theorem as shown in (2.1). Division rings and more generally noetherian rings are obviously F-noetherian. 
\end{example}
\begin{remark}\leavevmode
\begin{enumerate}[(i)]
\item F-noetherian rings and tightly F-noetherian rings are preserved under homomorphic images.
\item F-noetherian rings and tightly F-noetherian rings are preserved under direct limits. 
\item A subring (which contains the unity $1$) of a tightly F-noetherian ring $R$ is trivially a tightly F-noetherian ring. 
\item A subring of an F-noetherian ring may fail to be an F-noetherian ring.
\end{enumerate}
\end{remark}
\subparagraph{Example.}  Note that we can not take the field of fractions of a commutative polynomial ring P in infinitely many variables over a field because $P$ is commutative whence (tightly) F-noetherian. So let $R$ be any division hull of the free ring $\mathbb{Z}[ a, b]$ on two generators [L 1, (14.25)].
Then $R$ is noetherian being a division ring.  However, the subring $S : = \mathbb{Z}[ a, b]$  is not noetherian, for example, because the uniform dimension of $S$
as a right $S$-module is infinite as in [L 2, (1.31)]. But $S$ is generated by $\{a, b\}$.  Hence $S$ is not F-noetherian.

\begin{prop}\leavevmode
\begin{enumerate}[(i)]
\item If $R$ and $S$ are F-noetherian rings, then so is $R\times S$.   
\item If $R$ and $S$ are tightly F-noetherian rings, then so is $R\times S$. 
\end{enumerate}
\end{prop}
\begin{proof}\leavevmode
(i) let $X$ be a finite subset of $R\times S$.  Then $X$ is contained in $A\times B$ where $A$ and $B$ are finite subsets of $R$ and $S$ respectively. Since $R$ is F-noetherian, $A$ is contained in a noetherian subring $R'$ of $R$. Similarly, $B$ is contained in a noetherian subring $S'$ of $S$. Now $ R'\times S'$ is a noetherian subring containing $X$.  Hence $R\times S$ is F-noetherian.
\\ (ii)  This follows by the evident modification of the proof in part (i) and the following exercise. \end{proof}
\paragraph{Exercise} If $A$ is a subdirect product of two noetherian rings $R$ and $S$, then $A$ is noetherian. \emph{Recall that a subdirect product of rings $R$ and $S$ is a subring of $R\times S$ which projects surjectively onto each factor}.

\begin{proof} First we show that $S$ is a noetherian A-module as follows.  $S$ can be viewed as an $A$-module via the projection $g: R\times S \rightarrow S$  by $a.s= g(a)s$.� Let $W$ be a sub $A$-module of $S$.� For all $s\in S$ and $w\in W$, there exists $a\in A$ (since $A$ is a subdirect product of $R\times S$), such that� \begin{equation}\label{(*)}s.w= g(a).w = a.w.\end{equation}  So $W$ is an $S$-module. But $S$ is noetherian. Hence $W$ is a finitely generated $S$-module.  Consequently, $W$ is a finitely generated $A$-module by equation (\ref{(*)}). Similarly $R$ is a noetherian $A$-module. Hence $R\times S$ is a noetherian $A$-module, so every ideal $W$ of $A$ (which is an $A$-submodule of $R\times S$) is a finitely generated $A$-module.� Hence $A$ is noetherian.
\end{proof}
\begin{remark} Let $R$ be an F-noetherian domain.  Then $R$ is an Ore domain.
\end{remark}
\begin{proof} Let  $aR$ and  $bR$ be non-zero submodules. Then $a$ and $b$ (with $1$) are contained in a noetherian subring $S$. Since $S$ is a noetherian domain, $S$ is an Ore domain [L 2, (10.23)]. Hence $aS$ and $bS$ have non-zero intersection.  Consequently, $aR$ and $bR$ have non-zero intersection, so $R$ is a right Ore domain. Similarly, $R$ is a left Ore domain.
\end{proof}

\begin{prop}\label{multi} Let $R$ be an F-noetherian ring and let $S$ be a multiplicative subset of $R$. 
\begin{enumerate}[(i)]
\item If $S$ is central in $R$, then the localization ring $RS^{-1}$ is F-noetherian.  
\item Suppose there exists a finite set $Q$ in $R$ such that        
\[sR'=R's\]
for every element $s \in S$ and every subring $R'$ of $R$ containing $S \cup Q$.  
Then $S$ is a denominator set for $R$ and the localization ring $RS^{-1}$ is F-noetherian. 
\item Suppose there exists a finite set $Q$ of $R$ such that 
\[sR'=R's\]
for every element $s \in S$ and every \textbf{finitely generated} subring $R'$ of $R$ containing 
$S \cup Q$. If $R$ is also tightly F-noetherian, then $S$ is a denominator set of $R$ and the localization ring $RS^{-1}$ is $F$-noetherian. 
\end{enumerate}
\end{prop}

\begin{proof} \leavevmode
\begin{itemize}
\item[(i)] Let  $X:=\{x_1=\frac{a_1}{s_1}, \dots,  x_n=\frac{a_n}{s_n}\}$ be a subset of $RS^{-1}$. Then $\{a_1, \dots a_n; s_1, \dots, s_n \}$ is contained in a noetherian subring $R'$ of $R$.   Let  $S'= R' \cap S$. Now $S'$ is a denominator set in $R'$ since $S$ is central in $R$, 
and $R'$ is noetherian. Hence the localization $R'(S')^{-1}$ is noetherian [G-W, Cor. 9.18 ]. But $X$ is contained in $R'(S')^{-1}$.  Thus $X$ is contained in a noetherian subring of $RS^{-1}$.

\item[(iii)] For convenience, we shall prove (iii) before (ii). Let  $S_0$ be any finite set of generators for $S$. Then the proof of part (i) can be applied if we can check that $S'$ is a denominator set of $R'$ and $S$ is a denominator set of $R$. The fact they are both Ore sets follows from our assumption (*) that $sR'=R's$ for every element $s \in S$. But every Ore set in a noetherian ring is a denominator set  [G-W, Prop. 9.9].   
Hence $S'$ is a denominator set of $R'$. Now we prove that $S$ is right reversible in $R$. 
Suppose  $sr=0$ where $s \in S$ and $r \in R$. Then  $S_0 \cup (r \cup Q)$  \textbf{generates} a finitely generated noetherian subring $S(r)$ of $R$ (since $R$ is tightly F-noetherian). Now $S$ is a Ore set of $S(R)$ by (*) and $S(R)$ is noetherian.  As above, $S$ is a right denominator set in $S(R)$, and thus $S$ is right reversible in $S(R)$. 
Since $\{s, r\} \subset S(R)$, we have $sr=0$ (where $s \in S$ and $r \in S(R)$) implies that there exists $s^* \in S(R)$  such that $rs^*=0$.
Similarly, we prove that $S$ is left reversible in $R$.  Hence $S$ is a denominator set of $R$ as well.

\item[(ii)] For this part, the proof  in part (iii) applies except that now $S_0 \cup (r \cup Q)$ is contained in  a noetherian subring $S(r)$ of $R$ (since $R$ is only F-noetherian). But that is fine because our hypothesis now is that every element of $S$ normalizes every subring of $R$. This proves Proposition \ref{multi}.
\end{itemize}
\end{proof}

\begin{corr}\label{3.9} Let $R= K[x_1, \dots, x_n]$ be a finitely generated ring over a ring $K$ where each $x_i$ commutes with the elements of $K$. Let $S$ be finitely generated multiplicatively closed subset of $R$ where S is in the centralizer of $K$.
Suppose $K$ contains a finite subset $Q$ such that for all $s \in S$ 
and for each $x_i$, $sx_i= q_i(s)x_i s$ for some unit $q_i(s)$ in the subring $[ Q]$ generated by $Q$.  

If $R$ is F-noetherian, then $S$ is a denominator set for $R$ and the localization ring $RS^{-1}$ is F-noetherian.
\end{corr}

\begin{proof}
This follows very easily from part (ii) proved above as follows.  Let $R'$ be any subring of $R$ containing both $S$ and $Q$.  Let $s \in S$ and let $r' \in R'$. For illustration, say  $sr'= s(kx_1x_3x_2x_1) \in sR'$ for some $k \in K$.  Recall our hypotheses that each $s x_i= q_i(s) x_i s$ for some $q_i(s)$ in the subring  $[ Q]$ generated by $Q$,  
and that the centralizer of $K$ contains both $S$ and each $x_i$. 
Hence $ s r'= s(kx_1x_3x_2x_1) = q (kx_1x_3x_2x_1) s = qr' s$ for some  $q \in [ Q]$. 
Thus $sr' \in R's$ since $R'$ contains $Q$.  This illustrates that $sR'\subset R's$.  
Similarly $R's \subset sR'$ since each $q_i(s)$ in $s x_i= q_i(s)x_i s$ is a unit in $[ Q]$. Consequently,  $sR'= R's$, so we can apply part (ii) proved above to obtain our result.
\end{proof}

\vspace{2em}
\section{Quantum Iterated Ore Extensions and Skew Laurent Extensions} \label{quant}
\begin{secc} Recall that a \emph{skew polynomial ring} (in one variable) $R=S[x; f, d]$ is the ring of polynomials in $x$ with left coefficients in $S$, i.e, $R$ is the free (left) $S$-module with basis 
$\{x^n \mid n=0,1,2, \dots \}$ such that
\begin{enumerate}[(i)]
\item $xs=f(s)x +d(s)$ for all $s \in S$,
\item $f$ is an automorphism of $S$ (sending 1 to 1), and
\item $d$ is a (left) $f$-derivation of $S$.  That is, $d: S \rightarrow S$ is an additive map such that  $d(xy)=f(x)d(y)+d(x)y$  for  $x,y \in S$.
\end{enumerate}
Such skew polynomial rings $S[x; f, d]$ are also called \emph{Ore extensions} in which $f$ is an automorphism of $S$ (since $f$ is only assumed to be an endomorphism of $S$ for general Ore extensions).  If $f$ is the identity, $S[x;f, d]$ is written simply as $S[x;d]$. 

Recall that if $S$ is a noetherian ring, then so is the skew polynomial ring $S[x; f, d]$. 
This theorem is the \textbf{Hilbert Basis Theorem} for skew polynomial rings. For more details, see  
[G-W, p. 13]  (or [M-R, Thm. 1.2.9], [B-G, pp. 8-9], [K, p. 19]).
\end{secc}
\begin{prop}\leavevmode
\begin{enumerate}[(i)]
\item Let $R=S[x; d]$ be a differential skew polynomial ring. 
If $S$ is F-noetherian, then $R$ may \emph{fail} to be F-noetherian.
\item Let $R=S[x; d]$ be a differential skew polynomial ring.  
Suppose $d$ is locally nilpotent (i.e, $d$ acts nilpotently on every element $s$ of $S$), 
If $S$ is F-noetherian, then $R$ is F-noetherian.  Moreover, if $S$ is tightly F-noetherian, then $R$ is directed F-noetherian.�
\end{enumerate}
\end{prop}
\begin{proof}\leavevmode
\begin{enumerate}[(i)]
\item We shall give an example. Let $S= \ZZ[x_0, x_1, \dots , x_n, \dots]$ �be the commutative ring in infinitely many variables. Let $R=S[t; d]$ be the differential Ore extension where $d$ is the 
$S$-linear derivation of $S$ such that� $d(x_i)=x_{i+1}$.  Then $R$ is not noetherian because 
$R=S[t; d]$ is a free $S$-module and $S$ is not noetherian.  

Now $\{x_0 , t\}$ and $1$ generate $R$ as a ring. Hence $R$ is not F-noetherian although $S$ is F-noetherian since $S$ is commutative.
\item This is a special case of Theorem \ref{4.3} as shown next.
\end{enumerate}
\end{proof}
\begin{thm}\label{4.3}
Let $R := S[x_1; f_1, d_1][x_2; f_2, d_2] \dots  [x_n; f_n, d_n]$ be an iterated Ore extension over a ring $S$ such that
\begin{enumerate}[(1)]
\item for each $j>i$,��$(f_j)(x_i) =� q_{ji}(x_i)$ for some unit $q_{ji}�\in S$,
\item each $f_i$ is the identity on $S$, 
\item each $d_i(S) \subset S$, and each $d_i$ is locally nilpotent on $S$.
\end{enumerate}
Then
\begin{enumerate}[(i)]
\item If $S$ is F-noetherian, then so is $R$.
\item If $S$ is tightly F-noetherian, then $R$ is directed F-noetherian.
\end{enumerate}

We assume no restrictions on each $d_j (x_i)$  for $j>i$.
\end{thm}
\begin{proof}\leavevmode
\begin{enumerate}[(i)]
\item First we note that since $R$ is an iterated Ore extension of $S$, every element of $R$ is an $S$-sum of the standard PBW basis in all $x_i$, namely the monomials 
$(x_1)^{e_1},  x_2^{e_2}  \dots (x_n)^{e_n}$.

Given any finite subset of $R$, we collect carefully all ``relevant" constants relative to $B$ in $S$.
Such finitely many constants are contained in a noetherian subring $S^*$ of $S$. Then we consider the noetherian subring  $R^*= S^*[x_1; f_1, d_1][x_2; f_2, d_2] \dots  [x_n; f_n, d_n]$ which will contain the given finite subset of $R$.

For clarity, first \textbf{\emph{we assume that each $d_i$ is zero on $S$}}. Thus each $x_i$ commutes the elements of $S$ (because each $f_i$ is the identity on $S$). Now let $B$ be any finite subset of $R$. Then each $b \in B$ can be written as an $S$-sum of the above standard PBW basis of $R$. Let $B'$ be set of all $S$-coefficients of all $b \in B$ appearing in S. �To $B'$, we add all $q_{ji}$ and their inverses together with 

(*)  the coefficients appearing in $S$ of all $\{d_j(x_i) \vert$ for $j>i\}$ in terms of the standard PBW basis of $R$.

Let $B''$ be the resulting finite subset of $S$. Since $S$ is F-noetherian, $B''$ is contained in a noetherian subring $S^*$ of $S$.  It is easy to check that $R^*=S^*[x_1][x_2; f_2, d_2] \dots  [x_n; f_n, d_n]$ is a \emph{well-defined} iterated Ore extension. Finally, this subring of $R$ is noetherian and evidently containing the given finite set $B$ of $R$.
\begin{description}
\item[Case 1:] In general: recall that $f_i$ is the identity on $S$, each $d_i(S) \subset S$, and $d_i$ are commuting locally nilpotent on $S$. Now the proof in the general case is very similar to the proof in the above special case (with each $d_i=0$) except that we need to enlarge the finite set $B''$  to  
$D(B'')$ which is the \emph{union} of  $B''$ with the following set
$((d_1)^{e_1}(d_2)^{e_2} \dots (d_n)^{e_n})(x) $ 
with $x \in B''$ and each $e_i \in \mathbb{N}$. 
Note that $D(B'') \subset S$ since each $d_i(S) \subset S$.
Moreover, $D(B'')$ is a finite set because the $d_i$ are commuting locally nilpotent derivations on $S$. Then, as in the special case above, we take any noetherian subring $S^*$ of $S$ containing $D(B'')$. 
Now each $d_i$ preserves $S^*$ and each $f_i$ fixes the elements of $S^*$.
So we can form the \emph{well-defined} iterated Ore extension $R^*=S^*[x_1][x_2; f_2, d_2] \dots  [x_n; f_n, d_n]$ which is a noetherian subring of $R$ containing the given finite set $B$ of $R$.
\item[Case 2:] The proof is very similar to the proof in case 1 but by enlarging $B''$ to
$D(B'')$ which is the \emph{union} of $B''$ with following set:
\[\{d_{i_1}, d_{i_2} \dots d_{i_k} (x) \vert x \in B'', \text{ and each } i_k \in \{1, 2,  \dots , n\}\}.\]
Note our hypotheses in case 2 makes sure that $D(B'')$ is a finite subset of $S$.
\end{description}
\item By the proof of (i), for each finite set $B$, we have constructed the corresponding finite set $B''$ in $S$ such that if  $X \subset Y$ are finite subsets of $R$, we have $X'' \subset Y''$, whence 
$D(X'') \subset D(Y'')$.  Since $S$ is tightly F-noetherian, $D(X'')$ and $D(Y'')$ generate noetherian subrings $D(X'')^*$ and $D(Y'')^*$ in $R$. Moreover, we have $D(X'')^* \subset D(Y'')^*$. Hence $R$ is directed F-noetherian. This proves Theorem \ref{4.3}.
\end{enumerate}
\end{proof}
\begin{corr} Let  $R= S[x_1; f_1, d_1][x_2; f_2, d_2] \dots [x_n; f_n, d_n]$ be an iterated Ore extension over a ring $S$ such that
\begin{enumerate}[(1)]
\item for all $j>i$,�$(f_j)(x_i) =� q_{ji} (x_i)$ for some unit $q_{ji}$ in $S$, and
\item each $x_i$ commutes with the elements of S, 
\end{enumerate}
Then
\begin{enumerate}[(i)]
\item if $S$ is F-noetherian, then so is $R$, and 
\item if $S$ is tightly F-noetherian, then $R$ is directed F-noetherian. 
\end{enumerate}
\end{corr}
\begin{remark} For comparison, we record now Corollary 5.7 of the next section.

Let  $R= S[x_1; f_1, d_1][x_2; f_2, d_2] . . .  [x_n; f_n, d_n]$ be an iterated Ore extension over a ring $S$ such that
\begin{enumerate}[(1)]
\item for all $j>i$,�$(f_j)(x_i) =� q_{ji} (x_i)$ for some unit $q_{ji} \in S$, and 
\item for all $i$, �$ x_i S  + S = S + S x_i$, and $S = A[z_1, \dots, z_m]$ (a finitely generated ring over a subring $A$) such that each $x_i$ commutes with the elements of $A$,  
\end{enumerate}
then, if $S$ is F-noetherian, then so is $R$.
\end{remark}
\begin{remark}\label{4.6} \leavevmode
\begin{enumerate}[(i)]
\item If $S$ is F-noetherian, then $R=S[x, x^{-1}; f]$ may fail to be F-noetherian.
\item Let $R=S[x, x^{-1}; f]$  be a skew-Laurent ring extension such that 
for all $s \in S$, $f(s)= q_s s$ for some unit $q_s \in S$ fixed under the automorphism $f$.
If $S$ is tightly F-noetherian, then $R=S[x, x^{-1}; f]$ is directed F-noetherian.
\end{enumerate}
\end{remark}
\begin{proof}\leavevmode
\begin{enumerate}[(i)]
\item Let $S=\mathbb{Z}[a_0, a_1, a_{-1},  a_2, a_{-2},\dots, a_n, a_{-n},\dots]$.
 Let $R=S[x, x^{-1}; f]$  (where $f$ is an automorphism) be the skew-Laurent ring such that $f(a_i)=a_{i+1}$. That is, $x. (a_i) =f(a_i) x$.  Moreover, $x^m (a_i)= f^{m}(a_i) x^i$ for all $m, i \in \mathbb{Z}$.  Since $x a_i x^{-1}= a_{i+1}$ and  $x^{-1} a_i x= a_{i-1}$,  it follows that $a_0,  x , x^{-1}$  generate $R$.  But $R$ is not noetherian since $R$ is a free $S$-module and $S$ is not noetherian. Hence $R=S[x, x^{-1}; f]$  is not F-noetherian while $S$ is F-noetherian being commutative.  
\item For each element $b$ of $R$, let $C(b)$ be the set of all $S$-coefficients appearing in writing $b$ as a polynomial in $x$ and $x^{-1}$.  Now, for every $s \in C(b)$, $f(s)=q_s s$ for some unit $q_s \in S$ fixed under $f$. Let $Q(b) = C(b) \cup \{q_s \vert s \in C(b)\}$ which is a finite subset of $S$.

Now let $B$ be a finite subset of $R$.  Let $Q(B)$ be the union of all $Q(b)$ defined above as $b$ varies over $B$.  Let $S^*$ be the subring generated by $Q(B)$. 
Since for all $s \in S(B)$, $q_s$ is fixed under $f$, and $f^{-1}(s)=(q_s)^{-1}s$, it follows that $f(S^*) \subset S^*$  and $f^{-1} (S^*) \subset S^*$  Hence $f$ restricts to an automorphism of $S^*$.
But $S^*$ is noetherian since it is generated as a ring by a finite subset of $R$ which is tightly F-noetherian.  Hence the skew-Laurent subring $R^*=S^*[x, x^{-1}; f]$ is noetherian  [B-W, Thm. 1.17].
Finally, it is easily seen that $R^*=S^*[x, x^{-1}; f]$ contains the given finite subset $B$ of $R$.
This proves Remark \ref{4.6}.
\end{enumerate}
\end{proof}
\begin{prop}\label{4.7} Let $R=S[x_1, (x_1)^{-1}; f_1] [x_2, (x_2)^{-1}; f_2] \dots [x_n, (x_n)^{-1}; f_n]$ 
be an iterated skew-Laurent ring extension such that 
\begin{enumerate}[(i)]
\item for each $ j>i$, $(f_j)(x_i) =� q_{ji} x_i $ for some central unit $q_{ji}$ �in $S$, and
\item  each $f_i$ is the identity on $S$.
\end{enumerate}
If $S$ is F-noetherian, then so is $R$.
\end{prop}
\begin{proof} The proof is quite similar to the proof of Theorem \ref{4.3} (with no $d_i$s) as follows.
Briefly, for each finite subset $B$ of $R$, we collect all ``relevant" coefficients in $S$ from writing each element of $B$ as a polynomial in $x$ and $x^{-1}$.  To such finite coefficients in $S$, we add all  $q_{ji}$ and their inverses. The resulting finite subset of $S$ is contained in a noetherian subring $S'$ of $S$.   Now $R'=S'[x_1, (x_1)^{-1}; f_1] [x_2, (x_2)^{-1}; f_2] \dots [x_n, (x_n)^{-1}; f_n]$ is a well-defined iterated skew-Laurent ring extension. Moreover, $R'$ is noetherian by [B-W, Thm. 1.17] and $R'$ contains the given finite set $B$.  Hence $R$ is F-noetherian. This proves Proposition \ref{4.7}.
\end{proof}

\vspace{2em}
\section{Skew Quantum Ring Extensions} \label{skew}
\begin{defin} Following Passman in  [P, p. 180], a ring extension $R$ of a ring $S$ is called an almost centralizing extension of $S$ if $R=S[x_1, x_2, \dots, x_n]$ is generated as a ring over $S$ by finitely many elements ${x_1, x_2, \dots, x_n}$ such that
\begin{enumerate}[(i)]
\item each $x_i$ commutes with each element of $S$, and
\item for all $ i, j$, $[x_j,  x_i]  \in S +Sx_1+ \cdots+ Sx_n$.
\end{enumerate}
However, in [M-R, 8.6.6], the definition of an almost centralizing extension is more general in the sense that $R$ is generated as a ring over $S$ by finitely many elements $\{x_1, x_2, \dots, x_n\}$ where 
\begin{enumerate}[(i)]
\item $s x_i-x_i s  \in  S$ for each $s \in S$, and each $i$, and 
\item for all $i, j$, $[x_j,  x_i] \in S+Sx_1+ \cdots+ Sx_n$.
\end{enumerate}
But we shall stick to the first definition by Passman.
\end{defin}
\begin{thm}\label{5.2} Let $R=S[x_1, x_2, \dots, x_n]$  be an almost centralizing extension of a ring $S$ in the sense that $R$ is generated as a ring over $S$ by finitely many elements $\{x_1, x_2, \dots, x_n\}$ such that 
\begin{enumerate}[(i)]
\item each $x_i$ commutes with the elements of $S$.
\item for all  $i, j$,  $[x_j,  x_i]  \in S +Sx_1+ \cdots+ Sx_n$.
\end{enumerate}
If $S$ is $F$-noetherian, then so is $R$. 
Thus, all iterated almost centralizing extensions of an F-noetherian ring are F-noetherian.
\\ More generally, condition (ii) can be generalized to the following condition.
\\ (ii)'  For all $j > i$, there exist units $q_{ji}  \in S$ such that
\[(x_j x_i  - q_{ji} x_i  x_j)   \in S +Sx_1+ \cdots+ Sx_n \]
\end{thm}
Theorem \ref{5.2} is a special case of Theorem \ref{5.4} case 2 part (ii) as shown below.

\begin{lem}\label{5.3} Let $R=S [ x_1, x_2, \dots, x_n ] $ be a finitely generated ring over a subring $S$ such that 
\begin{enumerate}[(i)]
\item for all $i$, $S + S x_i = x_i S+ S$, and 
\item for all $j>i$, $(x_j x_i� - q_{ji}� x_i . x_j)� \in  S + S x_1+ \cdots + S x_n$ for some units $q_{ji}$ in $S$.
\end{enumerate}
Then, if $S$ is noetherian, then so is $R$.
\end{lem}
\begin{proof} The proof is almost verbatim as in the interesting proof of Thm. 6.14  by McConnel and Robson in [M-R, p. 29] since our added units $q_{ji}$ of $S$ do not hurt their proof.
\end{proof}

In Theorem \ref{5.4} below, we note that the noetherian part (5.4)(i) of case 3 is a slight generalization of Proposition I.8.17 in  [B-G, p. 77] which is one of the key steps in proving that the quantum group  
$\mathcal{O}_q(G)$ is noetherian. One difference in our generalization is that we do not assume that each $x_i$ commutes with the elements of $S$. We also note that the noetherian part 5.4(i) of case 2 below 
is a generalization of [L-R, Cor. 2.4] since we do not assume any PBW $S$-basis for $R$.

\begin{thm}\label{5.4}
Let $R = S [ x_1, x_2, \dots , x_n ] $ be a finitely generated ring over a subring $S$ satisfying the following two conditions.
\begin{enumerate}[(1)]
\item for all $i$, $x_i S  + S = S + S x_i$.
\item for all $j > i$,� there exist units $q_{ji} \in S$ such that we have one of the following cases.
\begin{description}
\item[Case 1:] $(x_jx_i� - q_{ji}� x_i x_j)� \in   S[x_1, \dots, x_{j-1}] + S x_j $
\item[Case 1':] $ (x_jx_i� - q_{ji}� x_i x_j)� \in   Sx_i + S[x_{i+1}, \dots , x_n]$ 
\item[Case 2:] $(x_j x_i� - q_{ji}� x_i  x_j)�  \in  S + Sx_1+ \cdots + Sx_n$

\item[Case 3:]  $(x_j x_i� - q_{ji} x_i x_j) = f + g$  where $f  \in S + Sx_1+ \cdots + Sx_n$ and if $i >1 $, $g$ is a finite sum of quadratic monomials  $s x_a x_b $ where $s \in S$ and either  $a$ or $b$ is at most $i-1$; however if $i=1$,  $g$ is a finite sum of quadratic monomials $s x_a x_b$ where $s \in S$ and $a=1$ and $b<j$ or vice versa ($b=1$ and $a<j$). 
\end{description}
\end{enumerate}

Then in all cases, we have the following.
\begin{enumerate}[(i)]
\item If $S$ is noetherian, then so is $R$.
\item Suppose $S = A[z_1, \dots, z_m]$ is a finitely generated ring over a subring $A$ such that each $x_i$ commutes with the elements of $A$. Then, if $S$ is F-noetherian, so is $R$.
\item Under the additional assumption in (ii), if $S$ is directed F-noetherian, then so is $R$.
\end{enumerate}
\end{thm}

\begin{proof}\leavevmode
\begin{enumerate}[(i)]
\item The proof in case 3 is a simple  modification of the interesting filtration method in [B-G, p. 77].
For convenience, we shall start with case 2.
\begin{description}
\item[Case 2:] Here we have  $x_j x_i� - q_{ji} x_i  x_j�  \in  S+ S x_1+ \cdots+ Sx_n $.
So we can take the standard filtration $A_d  := (S + S. x_1+\cdots+ S x_n)^d $ for $d \geq 1$, 
while $A_0 :=S$.  Let $y_i=x_i + A_0 \in A_1$.  Then,
since $S+ S x_1+ \cdots+ Sx_n \in A_1 $, it follows that gr$(R)= S [ y_1,y_2, \dots, y_n ] $ such that $y_j y_i� = q_{ji}� y_i  y_j$   for all $j>i$. Moreover,  $y_i S  = S�y_i$ for all $i$ since for all $i$, $x_i S  + S = S + S x_i$. 
Hence, since $S$ is noetherian, we can apply Lemma \ref{5.3} to obtain that $R$ is noetherian.

\item[Case 1:]  Here we have $x_j. x_i� - q_{ji}� x_i . x_j� = f_{ji} + s x_j \in  S [ x_1,\dots, x_{j-1} ] + Sx_j$ for some $s \in S$. So we choose the degree of every element $x_i$ by the formula  $d(x_i)= d_i= N^i$ where
$N$ is the maximum usual degree among all $f_{ji}$, and we choose $d(s)=0$ for all $s \in S$. 
Next we define a non-negative filtration on $R$ with $A_0 =S$ and for $d \geq 1$, $A_d$ is the set of all finite sums of products
$t_1 t_2 \dots t_r$  where each $t_i \in S \cup \{x_1,x_2, \dots, x_n\}$  such that  
\[d(t_1) + d(t_2) +\dots+d(i_r)  \leq  d. \]

Then $N. d_{j-1} \leq  d_i + d_j$ because it says $N. N^{j-1} \leq  N^i + N^j$.
Hence each $f_{ji}  \in   A _{d_i + d_j - 1}$.  Moreover,  $d_j \leq d_i + d_j$ . 
Consequently,  $f_{ji}+ sx_j  \in   A _{d_i + d_j - 1}$ for all $s \in S$. 
Again, as above, we end up with gr$(R)= S [ y_1, \dots, y_n ]$ such that $y_j y_i� = q_{ji}� y_i  y_j$   for all $j>i$, and 
$y_i S +S  = S�y_i +S$ for all $i$. Since $S$ is also noetherian, we can apply Lemma \ref{5.3} to obtain that $R$ is noetherian.

\item[Case 1':] Here we have $ x_j. x_i� - q_{ji}� x_i . x_j� =   s x_i + f_{ji}$ for some $s\in S$ and some 
$f_{ji} \in S [ x_{i+1}, \dots , x_n ]. $ 
In fact, case 1' is equivalent to case 1 by reversing the order of elements. That is,  using the permutation that sends $i$ to $n-i+1$ to get the ordering $\{x_n, x_{n-1}, \dots, x_1\}$. Or we can modify the proof for case 1 by choosing $d_i=  N^{n-i}$ where $N$ is the maximum of all usual degrees of $f_{ji}$.

\item[Case 3:] Our proof of for case 3 would be a simple modification of the proof of Prop. I.8.17 in [B-G, p. 77] where it is assumed versus our generalization that each $x_i$ commutes with the elements of $k$, and it is assumed that our $f=0$ and that 
$y_j y_i� - q_{ji} y_i  y_j=0 $ for $i=1$. Note that the authors in [B-G] used $i>j$ versus our notation $j>i$. Their trick was to assign a degree $d_i$ for each $x_i$ that leads to filtration such that 
 gr$(R)$, in their case, is generated as an $S$-algebra by the elements $y_1,y_2, \dots, y_n$ satisfying  
$y_j y_i� = y_i  y_j$ for all $j>i$. In fact their chosen degrees, which we shall keep, is $d(x_i)= d_i = 2^n-2^{n-i}$, and each $d(s)=0$ for all $s \in S$. In particular, their non-negative filtration starts with $A_0 =S$. 
However, we shall slightly modify the definition of $A_d$, for $d \geq 1$, to be all finite sums of words  $t_1 t_2 \dots t_r$  where each $t_i \in S \cup \{x_1,x_2, \dots, x_n\}$  such that
\[d(t_1) + d(t_2) +\dots+d(i_r)  \leq  d\]

This gives a non-negative filtration of $R$. The choice $d(x_i)  = 2^n-2^{n-i}$ makes our all the quadratic elements $s x_a x_b$ 
(appearing in $g$ in case $i>1$)(so $a$ or $b$ is at most $i-1$) have the property that $ d_a + d_b  \leq  d_i + d_j$ as shown in [B-G, p. 77]. Now we check our additional terms. In case $i=1$, the quadratic terms in $g$ have the property $d_1 + d_{j-1} \leq  d_1 + d_j$.  
Hence, for all $j>i$, $f_{ji}$ or $f \in  A _{d_i + d_j  - 1}$. 
Moreover, $S + Sx_1 + \cdots +Sx_n  \subset  A _{d_i + d_j  - 1}$ for all $j>i$   because One can easily show that 
$d(x_n) = d_n = 2^n - 1 < d(x_1) + d(x_2) = (2^n -2^{n-1}) + (2^n -2^{n-2})- 1$ because upon simplification, we need $3 ( 2^{2n-3}) -2^n +1 >0$ 
for all $ n \geq 2$ which is true since $\frac{3}{8} (x^2) - x +1 > 0$ for all $x \geq 2$.  
Hence gr$(R) = [ y_1,y_2, \dots, y_n ] $ such that  $y_j y_i� = q_{ji}�y_i  y_j$ for all $j>i$ and 
$ y_i S +S = S+ S� y_i $. Since $S$ is noetherian, we can apply Lemma \ref{5.3} to obtain that $R$ is noetherian.
\end{description}

\item Since each $x_i$ are not assumed to commute with the elements of $S$, then every element of the finitely generated ring $R= S[x_1, x_2, \dots , x_n]$ can be written as a finite sum of words $t_1, t_2, \dots, t_k$ where each $t_i \in S \cup \{x_1, x_2, \dots, x_n\}$. 
Such representation is not necessarily unique.  But we use the \emph{axiom of choice} 
to make a fixed choice of representations for all elements of $R$.  
Let $F=\{f_1,�\dots , f_m\}$ be any finite subset of $R$.   For all $i$, let $C(f_i)$ be the elements of $S$ appearing in the representation of $f_i$.  Similarly, for all $j>i$, let $C_{ji}$  be the elements of $S$ appearing in the representation of $f_{ji}$ depending on our case. Recall
\[x_j x_i - q_{ji} x_i  x_j  = f_{ji}\]
Let  $F'$  be the union of the following finite sets: $C(f_i)$ with all $C_{ji}$ with all $q_{ji}$. 
We need to add more constants from $S$ to $F''$ arising from the given condition: for all $i$, $x_i S  + S = S + S x_i$. Thus each $x_i  z_j  \in  Sx_i +S$ and each $ z_j x_i  \in  Sx_i +S$.  Hence  

\[x_i z_j  = b_{ij} x_i  + c_{ij}  \text{ and }    z_j x_i  = x_i d_{ij} + e_{ij}\]          

where all $b_{ij}, c_{ij}, d_{ij}, e_{ij}   \in S$.  

Now we enlarge the finite set $F'$ by adding all $b_{ij}, c_{ij}, d_{ij}, e_{ij}$, and let $X(F)$ be the resulting set. Note that $X(F)$ is a \emph{finite} subset of $S$ because we have finitely many $z_i$ in

$S =A [z_1, z_2,  \dots , z_m]$.  For (iii) in particular, we shall need the axiom of choice again to write all elements of $A[z_1, \dots, z_m]$ as words like $a_1 z_2 z_3 a_2 z_1 a_3$. With this in mind, let $X(F(A))$ be the set of all coefficients in $A$ appearing in the representations of all $x \in X(F)$ as words in  $A[z_1,\dots, z_m]$.  Since $S$ is F-noetherian, $X(F)(A)$ is \textit {contained} in a noetherian subring 
$A^*$ of  $S$.   Finally, let \[S^*= A^*[z_1,\dots, z_m]   \text {and let }    R^*= S^*[x_1, \dots  ,x_n]\]

To complete the proof, we shall need Lemma 5.5 proven below which shows that condition (1) is still satisfied. That is, for all $i$,  
$x_i S^*� + S^* = S^* + S^*�x_i$.
We also have condition (2) is satisfied.  Namely, for all $j > i$,  there exist units $q_{ji} \in S^*$ such that  all ``coefficients" of $S$ appearing in all the equations $x_j x_i  - q_{ji} x_i  x_j  = f_{ji}$ 
are in $S^*$ (where $f_{ji}$ depends on our four cases).  
Hence $R^*= S^* [x_1, x_2, \dots , x_n]$ is noetherian by part (i) and thereby proving (ii).�  
\item Now we are assuming that $S$ is directed F-noetherian.  
Our proof in (ii) with the axiom of choice (used twice), shows that for  finite subsets 
$F_1 \subset F_2$ in $R=S[x_1, x_2, \dots , x_n]$, we have constructed the corresponding finite sets in $S$ and then in $A$ (via where $S= A[z_1, \dots, z_m]$) such $X(F_1)(A)  \subset  X(F_2)(A) $.
These two finite sets are contained respectively in noetherian subrings $S_1^* \subset S_2^*$ since 
$S$ is directed F-noetherian. Thus  $S_1^*[x_1, x_2, \dots , x_n]  \subset  S_1^*[x_1, x_2, \dots , x_n]  $.  Hence $R$ is directed 
F-noetherian. This proves Theorem \ref{5.4} modulo the proof of Lemma 5.5 shown next.
\end{enumerate}
\end{proof}

\begin{lem}\label{lem5.5} 
Let  $R= S [ x_1, x_2, \dots, x_n ] $ be a finitely generated ring over its subring $S$ and suppose 
$S=  A [ z_1, \dots, z_m ] $ where each $x_i$ commutes with the elements of $A$, and for all $i$, 
$$�x_i S  + S = S + S x_i $$. 
In particular, we have  $x_i. z_j  = b_{ij}. x_i  + c_{ij}   \text { and }  z_j. x_i  = x_i. d_{ij}+ e_{ij}$ for some  $b_{ij}, c_{ij}, d_{ij}, e_{ij}  \in S$.  Let  $A^*$ be a subring of $A$ and let $S^*=  A^*  [ z_1, \dots , z_m ] $.  
If $S^*$ contains all $b_{ij}, c_{ij}, d_{ij}, e_{ij}$ defined above, then for all $i$, we have
\[x_i S^*� + S^* = S^* + S^* �x_i\]
\end{lem}
\begin{proof}
The elements of $S^*= A^*  [z_1, \dots, z_m]$  (which is a finitely generated ring over $A^*$) can be written as a finite sum of words  $y_1, y_2, \dots, y_k$  (say of length $k$) where each $y_i  \in  A^* \cup  \{ z_1, z_2, \dots , z_m\}$.  To check  each $x_i S^*� + S^* = S^* + S^* x_i$, we shall first prove the left-hand side, namely that each $x_i S^*� \subset S^* +S^*�x_i$  by induction on the length of words  in $S^*$ (while the right-hand side follows similarly).  Recall that 
\[x_i z_j  = b_{ij}. x_i  + c_{ij}   \text { and }  z_j. x_i  = x_i. d_{ij}+ e_{ij}\]
where  $b_{ij}, c_{ij}, d_{ij}, e_{ij}  \in S^*$.

By induction,  suppose the words in $S^*$ of length $k$ satisfy the desired (left) property that 
each $x_i S^*� \subset S* + S*�x_i$ . So let us take a word $w$ of length $k+1$.
Since $S^*=  A^* [z_1, \dots, z_m]$,  the word $w$ will have one of the following 4 forms. 
\[ a^*s^*, s^*a^*, z_js^*, s^*z_j  \]
where $ a^* \in A^* $ and $\text {length}(s^*)= k $

Recall that each $x_i$ commutes with elements of $A$, $A^* \subset A$, 
and, \textit {by induction}, each  $x_i s^*� \subset S^* + S^*�x_i$.

For form 1, each $x_i (a^*s^*)  = a^*x_i s^* \in a^*(S^*x_i +S^*)  \subset  S^*x_i +S^*$.
For form 2,  each $x_i (s^*a^*) =  (x_i s^*)a^* \in (S^*x_i +S^*)a^* = S^*a^*x_i + S^*a^* 
\subset S^*x_i +S^*$.  For form 3, each $x_i (s^*z_j)  \in   (S^*x_i +S^*)z_j = S^*(b_{ij}. x_i  + c_{ij}) + S^*z_j
 \subset S^*x_i +S^* $ because $S^*$ contains each $b_{ij}$ and $c_{ij}$.  For form 4, each $x_i (z_j s^*) = x_i. z_j  = (b_{ij} x_i  + c_{ij}) s^* \in b{ij} (S^*x_i +S^*) + c_{ij}) s^* \subset   S^*x_i +S^*$ because $S^*$ contains each $b_{ij}$ and $c_{ij}$.
 
Similarly, we can show that each  $ s^* x_i� \subset S^* + x_i S^*$  by using the equations
$z_j x_i  = x_i d_{ij}+ e_{ij}$  and $S^*$ contains each $d_{ij}$ and  $e_{ij}$ 
This proves Lemma \ref{lem5.5}.    
\end{proof}
\par  We remind the reader that the proof of Lemma 5.5 completes the proof of Theorem 5.4.

\begin{corr} \label{cases}
Let $R = S[x_1, x_2, \dots, x_n]$ be a finitely generated ring over a subring $S$ 
where each $x_i$ commutes with the elements of $S$, and 
for all $j > i$,�there exist units $q_{ji}� \in S$ such that we have one of the following cases.
\begin{description}
\item[Case 1:]  $(x_j x_i� - q_{ji}�x_i x_j)� \in   S[x_1, \dots, x_{j-1}] + S.x_j $

\item[Case 1':]  $ (x_jx_i� - q_{ji}� x_i x_j)� \in   Sx_i + S[x_{i+1}, \dots , x_n]$ 

\item[Case 2: ]   $(x_j x_i� - q_{ji}�x_i  x_j) � \in  S + Sx_1+ \cdots + Sx_n$

\item[Case 3:] 
$(x_j x_i� - q_{ji} x_i x_j) = f + g$  where $f  \in S + Sx_1+ \cdots + Sx_n$ and if $i >1 $, $g$ is a finite sum of quadratic monomials  $s x_a x_b $ where $s \in S$ and either  $a$ or $b$ is at most $i-1$; however if $i=1$,  $g$ is a finite sum of quadratic monomials $s x_a x_b$ where $s \in S$ and $a=1$ and $b<j$ or vice versa ($b=1$ and $a<j$). 
\end{description}

Then in all such cases, we have the following.
\begin{enumerate}[(i)]
\item If $S$ is noetherian, then so is $R$.
\item If $S$ is F-noetherian, then so is $R$.  
\item If $S$ is directed F-noetherian, then so is $R$.
\end{enumerate}
\end{corr}

\begin{corr} Let $R= S[x_1; f_1, d_1][x_2; f_2, d_2] \dots  [x_n; f_n, d_n]$
 be an iterated Ore extension over a ring $S$ such that
\begin{enumerate}[(i)]
\item for all $j>i$, $f_j(x_i) =�q_{ji} (x_i)$ for some unit $q_{ji}\in S$, and
\item for all $i$, $x_i S  + S = S + S x_i$ and $S = A[z_1, \dots, z_m]$ (is a finitely generated ring over a subring $A$) such that each $x_i$ commutes with the elements of $A$.
\end{enumerate}
Then, if $S$ is F-noetherian, then so is $R$. 
\end{corr}

\begin{proof}\leavevmode  Here we have  $(x_jx_i� - q_{ji}� x_i x_j) � \in   k[x_1, \dots, x_{j-1}]  $.
So we can apply Theorem 5.4 case 1 to obtain our result..
\end{proof}

\begin{thm}\label{5.8} Let $G= k[t_1,  \dots , t_m] [ x_1,  \dots , x_n ]= k[t_1,  \dots , t_m,  x_1,  \dots , x_n] $ be a finitely generated ring over a subring $k$ such that 
\begin{enumerate}[(1)]
\item For all $j > i$,�there exist units $p_{ji} \in k$  such that   
\[t_j t_i� - p_{ji}� t_i t_j� \in   k[t_1, \dots, t_{j-1}] + k t_j\]
\item For all $j > i$,�there exist units $q_{ji} \in k$  such that 
\begin{description}
\item[Case 1: ]   $(x_jx_i� - q_{ji}� x_i x_j)� \in   k[x_1, \dots, x_{j-1}] + kx_j $
\item[Case 2: ] $(x_j x_i� - q_{ji}� x_i  x_j)�  \in  k + kx_1+ \cdots + kx_n$
\item[Case 3: ]  $(x_j x_i� - q_{ji} x_i x_j) = f + g$  where $f  \in S + Sx_1+ \cdots + Sx_n$ and if $i >1 $, $g$ is a finite sum of quadratic monomials  $s x_a x_b $ where $s \in S$ and either  $a$ or $b$ is at most $i-1$; however if $i=1$,  $g$ is a finite sum of quadratic monomials $s x_a x_b$ where $s \in S$ and $a=1$ and $b<j$ or vice versa ($b=1$ and $a<j$). 

\end{description}

\item For all $j$ and $i$, there exist units $c_{ji}�\in k$ such that 
\[t_ix_j� - c_{ji}�x_j  t_i  \in  k + kt_1+ \cdots + k t_u+ kx_1+ \cdots + kx_v\]
where in case 1, $(u, v)=(i, j-1)$ or $(i-1, j)$;  in case 2, $(u, v)=(i, n)$; and in case 3, $(u, v)=(i-1, j)$.

\item for all $i$, $ k t_i +k = k + t_i k$,
\item for all $j$, $k x_j +k = k + x_j k$, and 
\item  $S = A[z_1, \dots, z_m]$ is a finitely generated ring over a subring $A$ such that each $x_i$ and each $t_j$ commutes with the elements of $A$.
\end{enumerate}

Then in all cases, we have the following.
\begin{enumerate}[(i)]
\item If $k$ is noetherian, then so is $G$.
\item If $k$ is F-noetherian, then so is $G$.  
\item If $k$ is directed F-noetherian, then so is $G$.  
\end{enumerate}
\end{thm}

\begin{proof}\leavevmode
\begin{enumerate}[(i)]
\item The proof is very similar to the proof of Theorem \ref{5.4} (i) with the following modifications.
\begin{description}
\item [Case 1:] We choose deg$(t_i)= N^i$ where $N$ is the maximum usual degree of all polynomials $f_{ji}$ in the relations $ t_j t_i� - p_{ji}� t_i t_j� = f_{ji}  \in   k[t_1, \dots, t_{j-1}] + k.t_j $.
We also choose deg$(x_i)= M^i$ where $M$ is the maximum usual degree of all polynomials $h_{ji}$ in the relations $x_jx_i� - q_{ji}� x_i x_j =h_{ji} � \in   k[x_1, \dots, x_{j-1}] + k.x_j $. Then the strategy of the proof in Theorem \ref{5.4} (i) works for the terms  containing  $ t_j t_i� - p_{ji}� t_i t_j$ and the terms containing $x_jx_i� - q_{ji}x_i x_j$.  So we still need to check the relations containing $ t_i x_j� - c_{ji}�x_j  t_i  $.   For this, assuming $(u, v) = (i, j-1)$, all we need is $ N^u + M^v= N^i+ M^{j-1} < N^i + M^j $ which is true.  Similarly, if $(u, v) = (i-1, j)$,
\item[Case 2: ] We go as in case 1 except that we choose deg$(x_i)=1$. Then the strategy of the proof in Theorem \ref{5.4} (i) works for the terms containing  $ t_j t_i� - p_{ji}� t_i t_j$  and the terms containing  $x_j x_i� - q_{ji}x_i x_j$.  So we still need to check the relations containing $t_i x_j� - c_{ji}�x_j  t_i  $. For this, all we need is 
$ N^u + 1= N^{i-1}+ 1 < N^i + 1$  which is true.

\item[Case 3: ] We go as in case 1 except that we choose deg$(x_i)=2^n - 2^{n-i}$ as done in the proof of Theorem \ref{5.4} (i). Then the strategy of the proof in that theorem works for the terms containing $ t_j t_i� - p_{ji}� t_i t_j$ and the terms containing  $x_j x_i� - q_{ji}x_i x_j$.  So we still need to check the relations containing  $ t_i x_j� - c_{ji}�x_j  t_i  $. For this, all we need is 
$ N^u + 2^n - 2^{n-v} = N^{i-1}+ 2^n - 2^{n-j}  <  N^i + 2^n - 2^{n-j} $  which is true.
\end{description}
\end{enumerate}

The proofs of part (ii) and (iii) are very similar to the proof of Theorem \ref{5.4} (ii) and (iii) and will be left to the reader.
This proves Theorem \ref{5.8}.
\end{proof}

\begin{thm}\label{5.9} Let $G= k[t_1,  \dots , t_m] [ x_1,  \dots , x_n ]= k[t_1,  \dots , t_m,  x_1,  \dots , x_n] $ be a finitely generated ring over a subring $k$ such that 
\begin{enumerate}[(1)]
\item For all $j > i$,�there exist units $p_{ji} \in k$  such that   
\[t_j t_i� - p_{ji}� t_i t_j� \in   k + k t_1 +  \cdots + kt_m\]
\item For all $j > i$,�there exist units $q_{ji} \in k$  such that 
\begin{description}
\item[Case 1: ]   $(x_jx_i� - q_{ji}� x_i x_j)� \in   k[x_1, \dots, x_{j-1}] + k.x_j $
\item[Case 2: ] $(x_j x_i� - q_{ji}� x_i  x_j)�  \in  k + kx_1+ \cdots + kx_n$
\end{description}

\item For all $j$ and $i$, there exist units $c_{ji}�\in k$ such that 
\[t_ix_j� - c_{ji}�x_j  t_i  \in  k + kt_1+ \cdots + k t_u+ kx_1+ \cdots + kx_v\]
where in case 1, $(u, v)=(i, j-1)$ or $(i-1, j)$; while in case 2, $(u, v)=(i, n)$.

\item for all $i$, $ k t_i +k = k + t_i k$,
\item for all $j$, $k x_j +k = k + x_j k$, and 
\item suppose $S = A[z_1, \dots, z_m]$ is a finitely generated ring over a subring $A$,  such that each $x_i$ and ecah $t_j$ commutes with the elements of $A$.  
\end{enumerate}

Then in all cases, we have the following.
\begin{enumerate}[(i)]
\item If $k$ is noetherian, then so is $G$.
\item If $k$ is F-noetherian, then so is $G$.  
\item If $k$ is directed F-noetherian, then so is $G$.  

\end{enumerate}  
\end{thm}

The proof  is very similar to the proof of Theorem \ref{5.8} with the following modifications 
on degrees where for all $i, j$, we let $\text{deg}(t_i)=1$ and $\text{deg}(x_j)=1$ as in the proof of Theorem \ref{5.8}.

\begin{remark} If a ring $R$ has a non-negative filtration whose associated graded ring gr$(R)$ is F-noetherian, the ring $R$ may fail to be F-noetherian.
The failure may occur even if gr$(R)$ is commutative.
\end{remark}
\subparagraph{Example.} Let $S= \mathbb{Z} [x_0, x_1, \dots , x_n, \dots]$ be the commutative ring in infinitely many variables.  Let $R=S[y; d]$ be the differential Ore extension where $d$ is the derivation of $S$ such that� $d(x_i)=x_{i+1}$. �So $yx_i �- x_i y = x_{i+1}$ for all $i$.
Then $R$ is not noetherian because $R=S[t; d]$ is a free $S$-module and $S$ is not noetherian. 
Moreover, $\{x_0 , t\}$ with $1$ generate $R$ as a ring.  Hence $R$ is not F-noetherian.
We make the standard filtration on $R$ induced by $\text{degree}(y) =2$ and 
$\text{degree}(x_i) =1$ for all $i$. More precisely, $A_0= \mathbb{Z}$ and 
$ A_m �= �(Zy + Zx_0 + Zx_1 + \cdots +Zx_n+� \cdots)^ m $ for $m \geq 1$.
Let $y^* = y+ A_0$,� and let �$x_i^*= x_i + A_0$.  Since ��$y x_i �- x_i y = x_{i+1}$ for all $i$,
It follows that� $y^*.x_i^*= x_i^*.y^*$ for all $i$ and all $x_i^*$ commute.
Hence gr$(R)$� is commutative so gr$(R)$ is F-noetherian even though $R$ is not F-noetherian.  

Finally, we show another filtration on $R$ such that each $A_m$  is finite dimensional and the associated graded ring gr$(R)$ is still commutative.  We let $d(x_i)=i$ and $d(y)=2$, and modify $A_m$ accordingly to be the $\mathbb{Z}$  finite sum of monomials  $t_{i_1} t_{i_2}\dots  t_{i_k} $ where each $t_i  \in \{y, x_0, x_1, \dots , x_n, \dots\}$ such that
\[d(t_{i_1}) + d(t_{i_2}) + \cdots  + d(t_{i_k})  \leq  m.\]
Then each $A_m$  is evidently finite dimensional, and the associated graded ring gr$(R)$ is still commutative even though $R$ is not F-noetherian.        

\vspace{2em}
\section{Applications and Examples} \label{App}

\textit{Throughout this section, we shall say that $x$ and $y$ almost commute in a ring $R$ if $xy=qyx$ for some unit $q\in R$. In this case, we also say that $q$ is the supporting constant of $xy$.}

For the convenience of the reader, we recall Corollary 5.6.

\paragraph{Corollary 5.6} Let $R = S[x_1, x_2, \dots, x_n]$ be a finitely generated ring over a subring $S$ 
where each $x_i$ commutes with the elements of $S$, and 
for all $j > i$,�there exist units $q_{ji}� \in S$ such that we have one of the following cases.
\begin{description}
\item[Case 1:]  $x_j x_i� - q_{ji}�x_i x_j� \in   S[x_1, \dots, x_{j-1}] + S.x_j $
\item[Case 1':]  $ (x_jx_i� - q_{ji}� x_i x_j)� \in   Sx_i + S[x_{i+1}, \dots , x_n]$ 
\item[Case 2:]   $x_j x_i� - q_{ji}�x_i  x_j� \in  S + Sx_1+ \cdots + Sx_n$

\item[Case 3: ]  $(x_j x_i� - q_{ji} x_i x_j) = f + g$  where $f  \in S + Sx_1+ \cdots + Sx_n$ and if $i >1 $, $g$ is a finite sum of quadratic monomials  $s x_a x_b $ where $s \in S$ and either  $a$ or $b$ is at most $i-1$; however if $i=1$,  $g$ is a finite sum of quadratic monomials $s x_a x_b$ where $s \in S$ and $a=1$ and $b<j$ or vice versa ($b=1$ and $a<j$). 

\end{description}

Then 
\begin{enumerate}[(i)]
\item If $S$ is noetherian, then so is $R$.
\item If $S$ is F-noetherian, then so is $R$.  
\end{enumerate}
\begin{example}
Consider the following $k$-algebra $R$ over a ring $k$ in the variables $x, y, z$ with relations
\begin{align*}� 
&w(x,y, z) =0 \\
&xy =�� q_1 yx� + f(y,z)�� + p_1 x \\
&xz =�� q_2 zx +� g(y,z) + p_2 x \\
&zy + yz =�� h(z)�� + p_3 y\\
\end{align*} for some polynomials, $f(y, z), g(y, z) \in k[y,z]$, $w(x,y,z) \in k[x,y,z]$, each $p_i \in k$, and each $q_i$ is a unit of $k$.� 

Then if $k$ is noetherian or $k$ is F-noetherian, so is $R$.  To see this, let $R^{+}$ be the $k$-algebra with the last three given relations. Then $R$ is a homomorphic image of $R^{+}$.� Since the noetherian and F-noetherian properties are preserved by homomorphic images, it suffices to work with $R^{+}$.� Then Corollary \ref{cases} applies to $R^{+}$ with the ordering shown in� $R^{+}=k[z, y, x]$.
\end{example}
\begin{example} 
Let $R$ be the $k$-algebra generated by the variables $x_i, y_i, z_i$, $1 \leq i \leq n$ such that any pair of variables almost commute (whose supporting constants are units in $k$) except for the following pairs where we have
\[x_i y_i + q_i y_i x_i=  f_i + p_i x_i,\] 
where $f_i \in k[y_1, \dots, y_n, z_1, \dots, z_n, x_1,  \dots, x_{i-1}]$
for some units $q_i \in k$, some $p_i \in k$.  Then if $k$ is noetherian or $k$ is F-noetherian, so is $R$.  This follows from Corollary \ref{cases} with the ordering shown in   
\[R=k[y_1, \dots, y_n, z_1, \dots, z_n, x_1, \dots, x_n]\]
\end{example}

\begin{example} \emph{The Woronowicz algebra over an F-noetherian ring k}
  
See [W] or [L-R, p. 1215].  Let $R$ be the $k$-algebra generated by $x, y , z$ subject to the relations
\begin{align*}
&xz-v^4zx= (1+v^2)x\\ 	
&xy-v^2yx= vz\\		
&zy- v^4 yz= (1+v^2) y
\end{align*}
where $v$ is a unit of $k$ (which is not a root of unity). Then if $k$ is noetherian or $k$ is F-noetherian, so is $R$. This follows from Corollary \ref{cases} (case 2).
\end{example}
\begin{example}\emph{A generalization of the multi-parametrized quantum Weyl algebra over an F-noetherian ring k.} 

For the multi-parametrized \emph{quantum Weyl algebra} over a field $k$, see  [L-R, p. 1218].  Let $R$ be the $k$-algebra generated by the variables $x_i, y_i$,   $1 \leq i \leq n$ such that any pair of variables almost-commute (whose supporting constants are units in $k$) except for the pairs $(x_i, y_i)$ for all $i$, we have the relations, for all $i$,
\[y_i x_i  -  q_i x_i y_i = f_i + p_i y_i\] 
and \\
$$ \text{case 1: } f_i  \in  k [x_1, y_1, \dots, x_{i-1} , y_{i-1}, \bf{x_i}] $$
$$  \text{case 2: } f_i   \in  k [\mathbf{x_i} , x_{i+1}, y_{i+1}, \dots, x_n, y_n] $$
for some units $q_i \in k$ and some elements $p_i \in k$.  

Then if $k$ is noetherian or $k$ is F-noetherian, so is $R$.  The first case follows from Corollary \ref{cases} with the ordering shown in  

                      \[R=k[x_1, y_1, \cdots,  x_i, y_i, \cdots , x_n,y_n]\]

The second case follows from \ref{cases} with ordering shown in

                      \[R=k[x_n, y_n, \cdots,  x_i, y_i, \cdots , x_1,y_1]\]
\end{example}
\begin{example} \emph{A generalization of the Hayashi algebra over an F-noetherian ring $k$}.

For the Hayashi algebra over a field $k$, see [H] or [L-R, p. 1217]. Let $R$ be the $k$-algebra generated by the variables $x_i, y_i, z_i$,    $1 \leq i \leq n$ such that any pair of variables almost-commute (whose supporting constants are units in $k$) except for the pairs $(x_i, z_i)$ where we have the relations  
                       
                                        \[ (z_i x_i + qx_i z_i) y_i  = 1 =   y_i (z_i x_i + qx_i z_i) \]

for some unit $q \in S$. Observe that the above relation implies that
                                   
                                         \[ (y_i)^{-1} = z_i x_i + q x_i z_i .\]

Then if $k$ is noetherian or $k$ is F-noetherian, so is $R$. This follows from Corollary \ref{cases} by considering  
\[R=k[(y_1)^ {\pm 1},  \dots,  (y_n)^ {\pm 1}, x_1, \dots, x_n,  z_1, \dots, z_n].\]
(Note that if $xy=qyx$, then $xy^{-1}=qy^{-1}x$ if $q$ commutes with $x$ and $y$) .   
\end{example}
\begin{example}\label{6.6} \emph{On some quadratic algebras}. 
 
Let $R$ be the $k$-algebra generated by the variables $x_i$ with $1 \leq i \leq 5$ such that any pair of variables almost-commute (whose supporting constants are units in $k$) except for the pairs $(x_5, x_4)$ 
where we have the relation
\[x_5 x_4  -  q x_4 x_5 = p_1 x_5x_3 + p_2x_5x_2 +p_3(x_3)^2 \]

for some unit $q \in k$ and some constants $p_i \in k$ for  $i=1, 2, 3$.
Now we can apply Theorem \ref{5.4} to see that if $k$ is noetherian or F-noetherian, then so is $R$.
\end{example}
\begin{example}\label{enveloping} \emph{The Quantum enveloping algebra of $U_q(sl(2,  k)$ over an F-noetherian ring $k$}.

$U_q(sl(2, k))$ is defined as the $k$-algebra generated by $x, y, z, z^{-1}$ with relations 
\begin{align*}
&xz = q^{-2}zx\\       
&yz = q^2 zy\\
&xy - yx = (q - q^{-1}) ^{-1} (z - z^{-1})
\end{align*}                       
where $q \in k $ and $q\neq  \pm 1$.  
Then if $k$ is F-noetherian, so is $ U_q(sl(2, k)) $.

We can see this in many ways. For example, we may apply Theorem \ref{5.4} on 
$U_q(sl(2, k)) = S[x, y]$ where  $S=k[z, z-1]$. 
More precisely, since $x$ and $y$ are in the centralizer of $k$, and $xz = q^{-2}zx$, it follows that $xS+S=S+Sx$.  Similarly, $yS+S=S+Sy$.   
Here we can apply Theorem \ref{5.4} to see that if $k$ is F-noetherian, so is $U_q(sl(2, k)) $.
For another proof, see (\ref{6.8}) below.
\end{example}
\begin{secc}\label{6.8}
We can also prove \ref{6.6} by applying Prop. \ref{multi} part (ii) on $U_q(sl(2, k)) = k[x, y, z]Z^{-1}$ as a localization with respect to $Z = \{z^i\text{ with }i \in \mathbb{N}\}$ 

More precisely, since $xz = q^{-2}zx$, $yz = q^2 zy$,  $q \in k$ and the centralizer of $k$ contains  $Z \cup \{x, y,z\}$, it follows that  $Z$ satisfies all hypotheses of Proposition \ref{multi} part(ii) 
with the notation $R= k[x, y, z]$ and $Q=\{z\}$. Thus, if $k$ is F-noetherian, so is $ U_q(sl(2, k)) $.    
\end{secc}
\begin{example} \emph{A generalization of $U_q(sl(2))(k)$}.   

Let $U$ be the $k$-algebra generated by $\{x, y, z, z^{-1}\}$ with the relations 
\begin{align*}
&xz =  q_1zx,\\
&yz = q_2 zy,\text{ and}\\
&xy - yx =   f(z, z^{-1} , x)  + q_3 y,
\end{align*}
where $f(z, z^{-1} , x)  \in  k[z, z^{-1}, x]$ and each $q_i$ is a unit of $k$.
Then, verbatim, as in the proof (\ref{6.6}), we can apply Theorem \ref{5.4} on $U=S[x, y]$  where $S= k[z, z^{-1}]$ 
to see that, if $k$ is noetherian or F-noetherian, so is $U$.
\end{example}
\begin{example}\label{deformation} \emph{Witten deformation of  $U_q(sl(2))$ over an F-noetherian ring}.    
E. Witten  (see [L-R], p. 1217) introduced and studied a 7-parameter deformation of  the universal enveloping algebra $U(sl(2, k))$ depending on a 7-tuple of parameters  $\mathbf{\xi}= (\xi_1,\dots, \xi_7)$ 
and subject to the relations 
\begin{align*}
&xz - \xi_1zx = \xi_2 x,\\
&zy - \xi_3yz = \xi_4 y\text{, and}\\ 
&yx - \xi_5 xy = \xi_6 z ^2 + \xi_7z  
\end{align*}
In the usual case where $k$ is a field, it is assumed that  $\xi_1, \xi_2,  \xi_3 \neq 0$.
 
In our general case we assume that $\xi_1, \xi_2,   \xi_3$ are units in $k$ and the resulting algebra is denoted by $W(\xi)(k)$.   If we apply Corollary \ref{cases} with the ordering shown in $W(\xi)(k)= k[x, z, y]$ or $W(\xi)(k) = k[z, x, y]$, we see that if $k$ is noetherian or  F-noetherian, so is $W(\xi)(k)$.
\end{example}

\vspace{2em}
\section{Some Quantum Groups over F-noetherian Rings.}\label{some}

Recall that we have discussed in (\ref{enveloping}) -(\ref{deformation}) some variants of the Quantum enveloping algebra of $U_q(sl(2,  k)$ over an F-noetherian ring $k$.  

In this last section, we give many examples of some quantum groups over F-noetherian rings.  

\begin{secc}\label{7.0} For convenience, a skew quantum  \textit {fully triangular} extension over $k$ shall mean a  $k$-algebra
$k[t_1,  \dots , t_m]$ generated by the variables  $\{t_1,  \dots , t_m\}$ such that for all $j > i$,�there exist units $p_{ji} \in k$  such that   
\[ t_j t_i� - p_{ji}� t_i t_j� \in   k[t_1, \dots, t_{j-1}] + k.t_j.\]
 \end{secc}
 \begin{secc}\label{7.1} \textbf{A generalization of the quantum group $\mathcal{O}_q(M_n(k))$  and its multi-parameter version.}

Recall that the quantum group $\mathcal{O}_q(M(n, k))$ is the k-algebra generated by all variables $x_{ij}$ with $1 \leq i,j \leq n $ such that each $2 \times 2 $ square matrix of four generators at the positions
$(i,j) , (i, j^\prime ), (i^\prime ,j),(i ^\prime , j ^\prime )$  where $i < i^\prime $  and $j < j^\prime $ generate a copy of 
the well-known quantum group $\mathcal{O}_q (M(2, k))$ of  $2 \times 2 $ matrices
[B-G, I.2.2].  \\
   
Let  $ G(\mathcal{O}_q (M_n(k)) =   k[t_1,  \dots , t_m,  x_{11},  \dots , x_{nn}]$ be the k-algebra withe evident variables such that 
\begin{enumerate}[(i)]
\item $k[x_{11},  \dots , x_{nn}] = \mathcal{O}_q(M_n(k)) $ (with $q$ being a central unit of $k$) is the $k$-algebra generated by the variables $x_{ij}$  with $1 \leq i, j \leq n$ and with  the standard relations of the quantum group $\mathcal{O}_q(M_n(k))$  of  $n \times n$ matrices.  \textit{See [B-K, p. 16]}. 

\item For all $j > i$,there exist units $p_{ji} \in k$  such that   
\begin{description}
\item[Case 1: ]  $ t_j t_i� - p_{ji}� t_i t_j� \in   k[t_1, \dots, t_{j-1}] + kt_j $
\item [Case 2: ]  $ t_j t_i� - p_{ji}� t_i t_j� \in   k +k t_1+  \cdots + kt_m $
\end{description}

\item Let $\{x_1,  \dots , x_{n^{2}}\}$ be the lexicographic ordering of $ \{x_{11},  \dots , x_{nn}\}$. 
For all $j$ and $i$, we also assume there exist units $c_{ji}$ in $k$ such that 
\[t_ix_j� - c_{ji}�x_j  t_i  \in  k + kt_1+ \cdots + k t_u+ kx_1+ \cdots + kx_v\]
\end{enumerate}
where in case 1,  $(u, v)=(i, j-1)$ or $(i-1, j)$, while in case 2, $(u,v)= (m,n^{2})$  
 
Note that our given ring  can be written as   $ G(\mathcal{O}_q (M_n(k)) =   k[t_1,  \dots , t_m] [  x_{11},  \dots , x_{nn} ]$

Then, if $k$ is noetherian, or F-noetherian, or directed F-noetherian, 
so is $ G(\mathcal{O}_q(M_n(k))) $. 
\end{secc}

\begin{proof}
 This follows immediately from Theorem \ref{5.8} and Theorem \ref{5.9} since the ring extension $\mathcal{O}_q(M_n(k))$ is fully triangular in the sense of (\ref{7.0}).  \emph{See [B-K, p. 16] for details}.
 
\emph{Note}.  Similarly, we may consider a generalization of the multi-parameter version of  
$\mathcal{O }_q(M_n(k))$ where the parameters are central units of $k$. (See [B-K, p.16]).
 \end{proof}

 \begin{secc}\label{7.2}
 \textbf{A generalization of the quantum groups $\mathcal{O}_q(\mathrm{SL}_n(k))$ and 
$\mathcal{O}_q(\mathrm{GL}_n(k))$}

Let  $G(\mathcal{O}_q(M_n(k)) )$ be as in (\ref{7.1}) except that we replace condition (3) with the stronger condition 
\begin{itemize}
\item[(3)'] each $t_i$ commutes with each $x_j$.
\end{itemize}
Let  $D_q $ be the quantum determinant of $\mathcal{O}_q(M_n(k)) $, so  $D_q$ is central in $G(\mathcal{O}_q(M_n(k)) )$  
\text {[B-K p. 17]}. 
Hence $D_q$  is also central in $G(\mathcal{O}_q(M_n(k)) )$ by condition (3)' above.
So we can form
\begin{align*}
G(\mathcal{O}_q(\mathrm{SL}_n(k)))  &=   G(\mathcal{O}_q(M_n(k)))  / [ D_q - 1]\\
G(\mathcal{O}_q(\mathrm{GL}_n(k)))  &=   G(\mathcal{O}_q(M_n(k))) [(D_q)^{ - 1}] 
\end{align*}

If $k$ is noetherian, or F-noetherian, or directed F-noetherian, 
then so is $ G(\mathcal{O}_q(\mathrm{SL}_n(k))) $ and  $ G(\mathcal{O}_q(\mathrm{GL}_n(k)))  $.  
 \end{secc}
 \begin{proof}
 For $ G(\mathcal{O}_q(\mathrm{SL}_n(k)))  =   G(\mathcal{O}_q(M_n(k)))  / [ D_q - 1] $, this follows immediately from (\ref{7.1}) since the noetherian, the F-noetherian, and the directed F-noetherian properties are preserved under homomorphic images (\ref{3.4}). 
 
For $ G(\mathcal{O}_q(\mathrm{GL}_n(k)))  =   G(\mathcal{O}_q(M_n(k))) [(D_q)^{-1}] $, this follows immediately from (\ref{7.1}) since  the noetherian, the F-noetherian, and the directed F-noetherian properties are preserved under localizations by a central subset (see \ref{multi}).
 \end{proof}
 \begin{secc} \textbf{A generalization of the quantum symplectic algebra $ \mathcal{O}_q(\mathrm{SP}_n(k)))$ over an F-noetherian ring $k$}.

Recall that the algebra  $\mathcal{O}_q(\mathrm{SP}_n(k)))$ is the $k$-algebra generated by the variables 
$x_i, y_j$  with $i, j = 1, \dots , n$  subject to the relations
\begin{align*}
&y_j  x_i = q^{-1}x_i y_j,  \hspace{3mm}  yj yi = q y_i y_j, \hspace{3mm}   1 \le i < j \le n\\
&x_j x_i = q^{-1} x_i x_j,  \hspace{3mm}   x_j  y_i = q y_i x_j, \hspace{3mm}  	 1 \le i < j \le n \\
&x_i y_i - (q^2)y_i x_i   = ( q^2 - 1)\sum_{t=1}^{i-1} q^{i-t} y_t  x_t,  \hspace{3mm} 	1 \le i \le n
\end{align*}
Then $\mathcal{O}_q(\mathrm{SP}_n(k)))$ is a triangular extension as in (\ref{7.0}) via the ordering shown in
 $k[x_1, y_1, x_2, y_2, \dots, x_n, y_n]$.   

Again, we can generalize $\mathcal{O}_q(\mathrm{SP}_n(k)))$ to $G(\mathcal{O}_q(\mathrm{SP}_n(k))))$ exactly as in (\ref{7.1}).

In conclusion, if $k$ is noetherian or F-noetherian or directed F-noetherian, then so is 
$G(\mathcal{O}_q(\mathrm{SP}_n(k))))$.  

 \end{secc}
 \begin{remark}
 Let  $S$  be a ring containing a field $k$ of characteristic 0, and let $q$ be a unit of $k$.  
Then, by extension of scalers,  we may form the quantum enveloping algebra $U_q(L)(S)$   
of a simple Lie algebra $L$  from the $k$-algebra $U_q(L)(k)$ and we can form the quantum group 
$\mathcal{O}_q(G)(S)$ of a connected semi-simple algebraic group G from the $k$-algebra  
$\mathcal{O}_q(G)(k)$  (see [B-G, chapters 1.6, 1.7]).
It is well-known that $U_q(L)(k)$  is noetherian [B-G, p. 55] and  that  $\mathcal{O}_q(G)(k)$ is also noetherian [B-G, p. 78]. 
\vspace {1mm} \\ 
\textbf {Problem}.  I expect a positive answer to the problem whether if $S$ is noetherian or F-noetherian or directed F-noetherian, then so are $U_q(L)(S)$  and $\mathcal{O}_q(G)(S)$. 
\end{remark}

\vspace{2em}
 \section{Examples of F-noetherian Matricial Rings and F-noetherian Group Algebras.}\label{matricial}
 In this last section, we give many matricial examples and many group algebras examples of F-noetherian rings where some examples are tightly F-noetherian  and some are non-tightly F-noetherian. 
 \begin{prop} Let $M_n(R)$ be the ring of $n\times n$ matrices over a ring $R$.
\begin{enumerate}[(i)]
\item \textbf{Examples of tightly F-noetherian rings.}  
If $R=\mathbb{Z}$,  then $M_n(R)$ is tightly F-noetherian. More generally, $M_n(R)$ is tightly F-noetherian if $R$ a ring which is a module-finite ring extension of either $\mathbb{Z}$ or $R$ is a prime field $\mathbb{Z}_p$. 
(Interesting examples of  such coefficient rings $R$ in (iv) are the finite fields or the ring of algebraic integers of a number field (since it has a finite basis over $\mathbb{Z}$)  [Ma, p. 30]).
\item \textbf{Example of an F-noetherian ring that is not tightly F-noetherian.}  
$M_2(\mathbb{Z}[x])$ is noetherian but not tightly F-noetherian.
\item \textbf{Example of an F-noetherian ring that may not be tightly F-noetherian.}  
If $R$ is an F-noetherian ring, then so is $M_n(R)$.
\end{enumerate}
\end{prop}
\begin{proof} \leavevmode
\begin{enumerate}[(i)]
\item Every subring $A$ (with 1) of $M_n(\ZZ)$  is noetherian because $A$ is a finitely generated $\ZZ$-module. Hence $M_n(\ZZ)$ is tightly F-noetherian.
 
In general, let $A$ be any subring of $M_n(S)$. Then $\ZZ \subset A \subset M_n(S)$.  Now $M_n(S)$ is evidently a finitely generated $\ZZ$-module and $A$ is a $\ZZ$-submodule of $M_n(S)$. Hence $A$ is noetherian. Hence $M_n(S)$ is tightly F-noetherian. The second case (with $\ZZ_p$ instead of $\ZZ$) is very similar.   
\item The matrix ring $M_2(\ZZ[x])$ is noetherian since it is a finitely generated $\ZZ[x]$-module and $\ZZ[x]$ is noetherian. But this ring is not tightly F-noetherian as follows. Take the following set $X$ of 18 upper triangular $2\times 2$ matrices listed by rows as: 
\[X:= \bigg\{\left( \begin{array}{cc}
a & b \\
0 & c\\
\end{array} \right) \bigg\vert  \hspace{0.5em}a, b \in \{0, 1, x\}\text{ and �}c \in \{0, 1\}\bigg\}.\] Then the subring generated by $X$ is exactly the subring $A$ of $2\times 2$ matrices of the form \[\left( \begin{array}{cc}
\ZZ[x] & \ZZ[x] \\
0 & \ZZ \\
\end{array} \right)\]�with no relation among the entries.  But $\ZZ[x]$ is not a noetherian $\ZZ$-module. So $A$ is not noetherian [L 1, (1.22)]. Hence $M_2(\ZZ[x])$ is not tightly F-noetherian (while it is F-noetherian). 
\item Let $A$ be a finite subset of $M_n(R)$. Then the finitely many entries of the elements of $A$ are contained in a noetherian subring $R'$ of $R$. Hence the subring $M_n(R')$ is noetherian (and containing the finite subset $A$).  In general, such matricial rings may fail to be tightly F-noetherian (see Example (x) below).
In general, $M_n(R)$ may fail to be tightly F-noetherian by (ii).
\end{enumerate}
\end{proof}

\begin{secc}  
\begin{enumerate}
\textbf{Examples from Group Algebras}  

\item \textit{Example of a tightly F-noetherian ring}.  Let $Z[G]$  be the group algebra of a locally finite group $G$ over $Z$. Then $Z[G]$ is tightly F-noetherian. To see this, let $X$ be a finite subset of $Z[G]$. Then $X$ involves finitely many basis elements of $G$ which generate a finite subgroup $G'$ of $G$. Now $Z[G']$ is a finitely generated $\mathbb{Z}$-module. Hence the subring generated by $X$ is noetherian being a sub $\mathbb{Z}$-module of  $\mathbb{Z}[G']$.  Hence $\mathbb{Z}[G]$  is tightly F-noetherian.

\item \textit{Example of an F-noetherian rings which is  not tightly F-noetherian}.
Let $K$ be a noetherian domain, and let $G$ be a free group with at least two generators.  Note that $K$ can be embedded in a division ring $D$ [L 2, (10.23)].  Now let $R$ be any division hull of the group $K$-algebra $K[G]$ which exists since $D[G]$  has a division hull where one such division hull is the Malcev-Neumann construction for an ordered group $G$ over $D$ because free groups are ordered ([L 1, (6.31), (14.24)] or [Passman 1]).  Then $R$ is F-Noetherian since it is a division ring. Now let $G'$ be the free subgroup on two generators $\{a, b\}$. Then the subring $\mathbb{Z}[G']$  is not noetherian because, for example, the free algebra $\mathbb{Z}[ X]$ with $|X| \geq 2$ is not noetherian as in [L 2, (1.31)]).  Here $\mathbb{Z}[G'] = \mathbb{Z}[ X, X^{-1}]$. Hence $R$ (which is any division hull of $K[G]$) is not tightly F-noetherian although $R$ is F-noetherian.

\item \textit{Example of an F-noetherian rings which is  not tightly F-noetherian}.
A slight variation of the preceding example is when $R$ is any division hull of the free $k$-algebra $K[ X]$ with $|X|\geq 2 $. (Note that $K[ X]$ is a proper subring of $K[G]$ where $G$ is the free group generated by $X$). Then $R$ is F-noetherian. Moreover, if  $\{a, b\}$  is a two-element subset of $X$, then as above, the free subring $\mathbb{Z}[ a, b] $ is not noetherian.  Hence $R$ (which is any division hull of $K[X]$) is not tightly F-noetherian although $R$ is F-noetherian.

\item \textit {Example}.  Let $K[G]$  be the group algebra of a locally finite group $G$ over an F-noetherian ring $K$. (Recall that a group $G$ is locally finite group if every finitely generated subgroup is finite).   (Interesting examples of $G$ are the finitary symmetric/alternating groups on an infinite set). Then $K[G]$ is F-noetherian. To see this, let $X$ be a finite subset of $K[G]$. Then we have 
the elements in $X$ involve finitely many basis elements of $G$ which generate a finite subgroup $G'$ of  $X$, and 
the elements in $X$ involve finitely many elements of $K$ which are contained in a noetherian subring $K'$ of $K$, such that $X$ is contained in $K'[G']$. Now $K'[G']$ is a finitely generated $K'$-module. Hence $K'[G']$ is a noetherian ring containing $X$. Hence $K[G]$ is F-noetherian.  

\item \textit {Example}. Let $K[G]$  be the group algebra of a polycyclic-by-finite group $G$ (for example $G$ is a finitely generated nilpotent group)  over an F-noetherian ring $K$.  Recall that a group $G$ is polycyclic-by-finite if $G$ has a polycyclic normal subgroup of finite index. Then $K[G]$ is F-noetherian as follows. Let $X$ be a finite subset of $K[G]$. Then the coefficients of the elements of $X$ form a finite subset of $K$ which is contained in a noetherian subring $K'$ of $K$. Now $K'[G]$ is a noetherian ring [J, p. 305] containing $X$.  Hence $K[G]$ is F-noetherian.
\end{enumerate}

\vspace {2mm}
The following Proposition is known; but in the absence of a reference, we give a proof.

\begin{prop}  The direct limits of noetherian rings are precisely the directed F-noetherian rings (as defined in section 3).

\begin{proof} \leavevmode
\textit.  First note that the direct limit of noetherian  rings are evidently the rings which are directed unions of noetherian subrings.
Now every directed F-noetherian ring (as defined in section 3)  is evidently a directed union of noetherian subrings.
\par  Conversely, let $R$  be a directed union of noetherian subrings. 
Then there exists a set  $S$  of noetherian subrings of $R$  which is directed (every two members $R_1, R_2$ of  $S$  are contained in some member  $R_3$  of  $S$) and whose union is  $R$. Now, for every 1-element set  \{a\} , we can choose  $R(a) \in S$  which contains  $a$.
Now suppose inductively that for some natural number  $m \geq 1$  we have have chosen, for every m-element subset  $A$  of  $R$,  an  
$R(A)\in S$ which contains  $A$.  For each $(m+1)$-element set  $A$,  the directedness of  $S$  allows us to choose an  $R(A)\in S$  which contains $R(B)$  for all of the finitely many proper subsets  $B$  of  $A$. Thus induction on  m completes the required construction.  
\end{proof}
\end{prop}
\end{secc}

\vspace{1 em}
\section{Open Problems} 
In this paper, our examples of F-noetherian rings can be shown to be directed F-noetherian (being built on directed F-noetherian rings).
We also know that F-noetherian rings are basic.  So we pose the following two problems.  
\begin {enumerate}
\item {\textbf {Problem}.} \textit{Find an example of an F-noetherian ring which is not directed F-noetherian. Or equivalently, find an example of an F-noetherian ring which is not a direct limit of noetherian rings}.  See (8.3).

\item {\textbf {Problem}.} \textit {Find an example of a basic ring which is not F-noetherian}.

\item We briefly recall Problem 7.5.    Let  $S$  be a ring containing a field $k$ of characteristic 0, and let $q$ be a unit of $k$.  
Then, by extension of scalers,  we may form the quantum enveloping algebra $U_q(L)(S)$   
of a simple Lie algebra $L$ and we can form the quantum group $\mathcal{O}_q(G)(S)$ of a connected semi-simple algebraic group G. (see [B-G, chapters 1.6, 1.7]).   I expect a positive answer to the question 
whether if $S$ is noetherian or F-noetherian or directed F-noetherian, then so are $U_q(L)(S)$  and $\mathcal{O}_q(G)(S)$.  
\end {enumerate}

\vspace{2em}
\textit{Nazih Nahlus,  Math. Dept.,  American University of Beirut,  Beirut, Lebanon}

\textit{Email:  nahlus@aub.edu.lb}

\newpage
\begin{thebibliography}{9}
\bibitem[A-F-S]{A-F-S}   E. P. Armendariz, J. W. Fisher and R. L. Snider, 
\textit {On injective and surjective endomorphisms of finitely generated modules}, 
Comm. Algebra 6 (1978), 659-672.
\bibitem[B-G]{B-G}	K. A. Brown and K. R. Goodearl, \textit{Lectures on Algebraic Quantum Groups},
birkh{\"a}usser Verlag, 2002.
\bibitem[D]{D}     D.Z. Djocovic, \textit{Epimorphism of modules which must be isomorphism},
Canad. Math. Bull. 16 (4) (1973) 513-515
\bibitem[G]{G}   K.R. Goodearl, \textit{Surjective endomorphisms of finitely generated modules},
Comm. Algebra 15 (3) (1987), 589-609.
\bibitem[G-W]{G-W}  K. R. Goodearl and R. B. Warfield,  \textit{An Introduction to Non-commutative Noetherian Rings}, London Mathematical Society, ST 61, 2004.
\bibitem[Hag-V]{Hag-V}   A. Haghani and M.R. Vedadi, \textit {Modules whose injective endomorphisms
are essential}, J. Algebra 243 (2001) 765-779.
\bibitem[Hay]{Hay}   T. Hayashi,  \textit {$q$-analogues of Clifford and Weyl algebras. 
Spinor and oscillator representations of quantum enveloping algebras},
Comm. Math. Phys.  127:  129-144, 1990.
\bibitem[J]{J}	A. V. Jategaonkar, \textit{Localization in Noetherian Rings}, London Mathematical Society, LNS 98, Cambridge Univ. Press, 1986 
\bibitem[Kap]{Kap} 	I. Kaplansky, \textit{Commutative Rings}, The University of Chicago Press, 1974
\bibitem[Kas]{Kas}	Kassel, C., \textit{Quantum Groups}, Graduate texts in Mathematics, Springer Verlag 155, 1995.
\bibitem[L 1]{L1} 	T. Y. Lam, \textit{A First Course in Non-commutative Rings}, GTM 131, Springer, 2001
\bibitem[L 2]{L2} 	T. Y. Lam, \textit{Lectures on Modules and Rings}, GTM 189, Springer, 1999
\bibitem[L 3]{L3}	    T. Y. Lam, \textit{Exercises in Modules and Rings}, Problems Books in Mathematics, Springer, 2007
\bibitem[L-R]{L-R}   O. Lezama and A. Reyes, \textit{Some homological Properties of Skew PBW Extensions}, Comm. Algebra,  42:3, 1200-1230, 2014.
\bibitem[Ma]{Ma}	D. Marcus, \textit{Number Fields}, Springer-Verlag, 1977 
\bibitem[M-R]{M-R} 	J. C. McConnel and J. C. Robson,  \textit{Noncommutative Noetherian Rings}, Graduate Studies in Mathematics, Vol. 30, 2001 (pp.  42, 339).
\bibitem[P]{P} D.S. Passman,  \textit{A Course in Ring Theory}, Wadsworth and Brooks/Cole, 1991
\end{thebibliography}
\end{document}